\documentclass[reqno]{amsart}
\usepackage{hyperref}
\usepackage{amsfonts}
\usepackage{amsthm,amsmath}
\usepackage[integrals]{wasysym}
\usepackage{enumerate}
\usepackage{fancyhdr}
\usepackage{amssymb}
\usepackage{chemarrow}
\usepackage{tikz}
\usepackage{mathtools}
\usepackage{mathrsfs}

\numberwithin{equation}{section}
\newtheorem{theorem}{Theorem}[section]
\newtheorem{lemma}[theorem]{Lemma}
\newtheorem{definition}[theorem]{Definition}
\newtheorem{proposition}[theorem]{Proposition}
\newtheorem{corollary}[theorem]{Corollary}
\theoremstyle{remark}
\newtheorem{remark}{Remark}
\allowdisplaybreaks

\newcommand{\with}{\quad\hbox{with}\quad}
\newcommand{\andf}{\quad\hbox{and}\quad}

\def\wh{\widehat}
\def\wt{\widetilde}

\def\cC{{\mathcal C}}
\def\cD{{\mathcal D}}

\def\cP{{\mathcal P}}

\def\du{\delta\!u}
\def\dv{\delta\!v}
\def\dB{\delta\!B}
\def\dJ{\delta\!J}

\def\div{ \hbox{\rm div}\,  }
\def\curl{ \hbox{\rm curl}\,  }

\def\N{{\mathbb N}}
\def\R{{\mathbb R}}

\def\eps{\varepsilon}
\def\La{\Lambda}
\begin{document}
\title[\hfilneg \hfil ]
{The global solvability of  the Hall-magnetohydrodynamics system in critical Sobolev spaces}

\author{RAPHA\"{E}L DANCHIN AND JIN TAN}


\begin{abstract}
We are concerned with the  3D incompressible Hall-magnetohydro\-dynamic system (Hall-MHD). 
Our first  aim is to provide the reader with an elementary proof of a global well-posedness result for small data with critical Sobolev regularity, in the spirit 
of Fujita-Kato's theorem \cite{Fk64}  for  the Navier-Stokes equations. 
Next,  we investigate  the long-time asymptotics of  global solutions 
of the Hall-MHD system that are in the Fujita-Kato regularity class. 
  A weak-strong uniqueness statement is also proven.  Finally, 
we consider the so-called 2$\frac12$D flows for the Hall-MHD system, and prove the global existence of strong solutions, assuming only that the  initial magnetic field is small.
\end{abstract}

\maketitle
\section{Introduction}
We are concerned with the following three dimensional incompressible resistive and viscous Hall-magnetohydrodynamics system:
\begin{align}\label{1.1}
&\partial_t{\mathnormal u}+u\cdot\nabla u+\nabla \mathnormal P=(\nabla\times \mathnormal B)\times \mathnormal B+\mu\Delta \mathnormal u\quad&\hbox{in }\ \R_+\times\R^3,\\\label{1.2}
&\div\mathnormal u=0\quad&\hbox{in }\ \R_+\times\R^3,\\\label{1.3}
&\partial_t{\mathnormal B}-\nabla\times(\mathnormal u \times \mathnormal B)+\eps\nabla\times((\nabla\times\mathnormal B) \times \mathnormal B)=\nu\Delta \mathnormal B
\quad&\hbox{in }\ \R_+\times\R^3, \\\label{1.4}
&(\mathnormal u, \mathnormal B)|_{t=0}=(\mathnormal u_{0}, \mathnormal B_{0})
\quad&\hbox{in }\ \R^3,
\end{align}
where the unknown functions  $\mathnormal u$, $\mathnormal B$ and $\mathnormal P$ represent the  velocity  field, the magnetic field and the scalar pressure, respectively.  The parameters $\mu$ and $\nu$ are the fluid viscosity and the magnetic resistivity, while the dimensionless positive number $\eps$ measures   the 
magnitude  of the Hall  effect compared to  the typical length scale of the fluid. 
For compatibility with \eqref{1.2},  we assume that $\div u_0=0$ 
and, for physical consistency,  since a magnetic field is a curl,
  we suppose  that $\div B_0=0,$ 
a property that is propagated by \eqref{1.3}. 
 \smallbreak
The above system is used to model the   
\emph{magnetic reconnection} phenomenon,  that cannot be explained by the  classical  MHD
system where   the \emph{Hall electric field}
$E_H:=\eps J\times B$ (here the current  $J$ is defined by $J:=\nabla\times B$) is neglected.
 The study of the Hall-MHD system has been initiated  by Lighthill  in \cite{Li60}. Owing to its  importance in the theory of space plasma,  like e.g. star formation, solar flares or geo-dynamo, it has received  
 lots of attention from physicists  (see  \cite{Ba01,Fo91,Hu03,Sh97,Wa04}).
\smallbreak
The Hall-MHD system has been considered in mathematics only rather recently.
In   \cite{Ac11},  Acheritogaray, Degond, Frouvelle and Liu  formally derived 
the Hall-MHD system both   from a two fluids system and  from a  kinetic model. Then, in \cite{Ch13}, Chae, Degond and  Liu showed the global existence of Leray-Hopf weak solutions as well as the local 
existence of classical solutions pertaining to data with large Sobolev regularity. 
Weak solutions have been further investigated 
by  Dumas and Sueur in \cite{Ds14} both for the Maxwell-Landau-Lifshitz system and for the Hall-MHD system. In \cite{Ch14}, blow-up criteria  and the global existence of smooth solutions  emanating from  small initial data have been obtained. In \cite{We161,We162}, Weng studies the long-time behaviour and obtained optimal space-time decay rates of strong solutions. More recently, \cite{Be16,Wz15,Wz19}  established the well-posedness of strong solutions with improved regularity conditions for initial data in sobolev  or Besov spaces. 

 Examples of smooth data with arbitrarily large  $L^\infty$ norms
giving rise to  global unique solutions have been exhibited  in \cite{Li19}.  Very recently, 
in \cite{Da19}, the authors of the present paper proved well-posedness 
in critical Besov spaces and pointed out that better results may be achieved 
if $\mu=\nu.$ 
\smallbreak
In order to explain what we mean by critical regularity, 
let us first recall how it goes for the following incompressible Navier-Stokes equations:
$$\left\{\begin{aligned}
&\partial_t{\mathnormal u}+{\rm{div}}(u\otimes u)+\nabla \mathnormal P=\mu\Delta \mathnormal u&\qquad\hbox{in }\ \R_+\times\R^3,\\
&\div\mathnormal u=0&\qquad\hbox{in }\ \R_+\times\R^3,\\
&u|_{t=0}=u_{0}&\qquad\hbox{in }\ \R^3
\end{aligned}\right.\leqno(NS)$$
where
$$\bigl({\rm{div}}(v\otimes w)\bigr)^{j}{:=}\sum_{k=1}^{3}\partial_{k}(v^{j}w^{k}).$$
Clearly, $(NS)$ is invariant for all positive $\lambda$ by the rescaling
\begin{equation}\label{eq:solutionNS}
(u,P)(t, x)\leadsto (\lambda u,\lambda^2 P) (\lambda^{2}t, \lambda x)\andf
u_0(x)\leadsto \lambda u_0(\lambda x).
\end{equation}
For that reason,   global well-posedness results for $(NS)$ in  optimal  spaces  (in terms of regularity)
may be achieved by means of contracting mapping arguments only if using 
norms for which $u_0$ and 
the solution $(u,P)$  have the above scaling invariance. 
\medbreak
 In contrast with the Navier-Stokes equations, 
 the Hall-MHD system does not have a genuine scaling invariance (unless
 if the Hall-term is neglected). 
 In a recent paper of ours \cite{Da19}, we observed that 
 some scaling invariance does exist if one considers the current function  $J=\nabla\times B$ 
 to be an additional unknown. Then, the scaling invariance of $(u,B,J)$ is  the same as  the velocity in \eqref{eq:solutionNS}. 
 
 In \cite{Da19}, we also pointed out that the Hall-MHD system better behaves if
 $\mu=\nu$ since, although being still quasi-linear, the Hall term disappears
 in the energy estimate involving the so-called  \emph{velocity of electron $v:=u-\eps J.$}
 Indeed, let us consider  $v$  as an additional unknown
 and look at the following \emph{extended formulation of the Hall-MHD system}:
\begin{equation}\label{1.100}\left\{\begin{aligned}
&\partial_t u-\mu\Delta u=B\cdot\nabla B-u\cdot\nabla u-\nabla \pi,\\
&\partial_t{\mathnormal B}-\mu\Delta B=\nabla\times(v\times B),\\
&\partial_tv-\mu\Delta v=B\cdot\nabla B-u\cdot\nabla u-\eps\nabla\times((\nabla\times v)\times B)\\
&\hspace{4cm}+\nabla\times(v\times u)+2\eps\nabla\times(v\cdot\nabla B)-\nabla \pi,\\
&\div u=\div B=\div v=0.
\end{aligned}\right.\end{equation}
That (redundant) system is still quasilinear but,  owing to 
\begin{equation}\label{v1}
(\nabla\times w\,|\,v)=(w\,|\, \nabla\times v),
\end{equation}
where $(\cdot\,|\,\cdot)$ denotes the scalar product in $L^2(\R^3),$
the most nonlinear term cancels out when performing an energy method, since
\begin{equation}\label{1.1100}
(\nabla\times((\nabla\times v)\times B)\,|\,v)=0.
\end{equation}
Our primary goal is to provide an elementary proof  of  a  Fujita-Kato type result 
for the Hall-MHD system in the spirit of the celebrated work \cite{Fk64} (see also \cite{Ch92}).
In our context, this amounts to proving that System \eqref{1.1}--\eqref{1.3}
supplemented with initial data $(u_0,B_0)$ such that  
$(u_0,B_0,v_0)$ is small enough in the critical homogeneous Sobolev space $\dot H^{\frac12}(\R^3)$
admits a unique global solution.  In passing, we will obtain 
some informations on the long time behavior of the solutions, similar to those
that are presented  for the incompressible Navier-Stokes equations in e.g. \cite[Chap. 5]{Ba11}.

 Our second purpose is to prove    a weak-strong uniqueness result 
  for the  Leray-Hopf weak solutions of Hall-MHD, namely
  that all weak solutions coincide with the unique Fujita-Kato solution 
  whenever the latter one exists. That result turns out to be less sensitive
 to the very structure of the system, and is valid for all values of $\mu,$ $\nu$ and $\eps.$ 
    
Finally,  as proposed by  Chae and Lee in \cite{Ch14}, 
we will consider the $2\frac12$D flows for the Hall-MHD system, 
that is, $3$D flows depending only on two space variables. 
This issue is well known for the incompressible Navier-Stokes equations 
 (see e.g. the book by  Bertozzi and Majda  \cite{BM}).
In our case, the corresponding system reads:
\begin{align}
&\partial_t{\mathnormal u}+\wt u\cdot\wt\nabla u+\wt\nabla\pi=\wt B\cdot\wt\nabla B+\mu\wt\Delta u&\hbox{in }\ \R_+\times\R^2,\label{1.1a}\\\label{1.2a}
&\widetilde{\div}\wt u=0&\hbox{in }\ \R_+\times\R^2,\\
&\partial_t{\mathnormal B}+\wt u\cdot\wt\nabla B+\eps\wt B\cdot\wt\nabla j-\eps\wt j\cdot\wt\nabla B=\nu\wt\Delta B+\wt B\cdot\wt\nabla u&\hbox{in }\ \R_+\times\R^2, 
\label{1.3a}\\
&(\mathnormal u, \mathnormal B)|_{t=0}=(\mathnormal u_{0},\mathnormal B_{0})
&\hbox{in}\ \R^2\label{1.4a},
\end{align}
where the unknowns $u$ and $B$ are functions from $\R_+\times\R^2$ to $\R^3,$
 $\wt u:=(u^1, u^2),$ $\wt{B}:=(B^1, B^2),$ $\wt\nabla:=(\partial_1, \partial_2),$ 
$\wt\div :=\wt\nabla\cdot,$
$\wt\Delta:=\partial_1^2+\partial_2^2$ and 
 $$   j:=\wt\nabla\times B= \left(  \begin{array}{c}
          \partial_2 B^{3} \\
         -\partial_1B^{3}\\
         \partial_1 B^{2}-\partial_2 B^{1}
          \end{array}
\right)\cdotp$$

 A small modification of the proof of \cite{Ch13}  
  allows to establish that for any initial data $(u_0, B_0)$ in $L^2(\R^2;\R^3)$   with $\wt\div \wt u_0=\wt\div \wt B_0=0,$ 
 there exists a global-in-time Leray-Hopf solution $(u, B)$
 of \eqref{1.1a}--\eqref{1.4a}  that satisfies:
\begin{multline}\label{energyineq}
\|u(t)\|_{L^2(\R^2)}^2+\|B(t)\|_{L^2(\R^2)}^2 
+2\int_0^t\bigl(\mu\|\wt\nabla u\|_{L^2(\R^2)}^2+\nu\|\wt\nabla B\|_{L^2(\R^2)}^2\bigr)\,d\tau\\
\leq
\|u_0\|_{L^2(\R^2)}^2+\|B_0\|_{L^2(\R^2)}^2.\end{multline} 
Whether that solution is unique and  equality is true in \eqref{energyineq} are open questions.
The difficulty here is that, unlike for the   $2\frac12$D Navier-Stokes equations or 
for the $2\frac12$D MHD flows
 with no Hall term, the two-dimensional system satisfied by the first
 two components of the flow is coupled with the equation satisfied
 by the third component,   through the term $\wt B\cdot\wt\nabla j-\wt j\cdot\wt\nabla B,$
 thus hindering any attempt to prove the global well-posedness for large data by 
 means of classical arguments. 
 
Our aim here is to take advantage of  the special structure of the 
system so as  to get a global well-posedness statement 
 for $2\frac12$D data such that \emph{only the initial magnetic field is  small}.
 Since it  has been pointed out in \cite{Ch14} that controlling
just  $j$ in the space $L^2(0,T;{\rm BMO}(\R^2))$ prevents blow-up 
of a smooth solution at time $T$ and because 
the space  $\dot{H}^1(\R^2)$ is continuously embedded in ${\rm BMO}(\R^2),$ 
it is natural to  look for a control on $j$ in the space $L^2(0,T; \dot{H}^1(\R^2))$ for all $T>0.$

For $2\frac12$D flows, the critical Sobolev regularity corresponds to 
$u_0$ in $L^2(\R^2)$ and to $B_0$ in $H^1(\R^2),$ 
and we shall see that, indeed, in the case of small data, one can establish 
rather easily a global well-posedness result at this level of regularity. 
For large data, owing to the coupling 
between  the velocity  and magnetic fields through the term $\wt\nabla\times(u\times B),$
we do not know how to achieve global existence at this level 
of regularity.
Then, our idea is to look at the equation 
satisfied by  the  vorticity $\omega:=\wt\nabla\times u,$ namely 
$$
\partial_t{\omega}+\wt\nabla\times(\omega\times u)=\wt\nabla\times(j\times B)+\mu\wt\Delta w.
$$
In the case  $\mu=\nu,$ the vector-field  $E:=\eps\omega+B$  thus satisfies
\begin{equation}\label{eq:E}
\partial_t E-\wt E\cdot\wt\nabla u+\wt u\cdot\wt\nabla E=\mu\wt\Delta E.
\end{equation}
From that identity, obvious energy arguments and \eqref{energyineq}, 
it is easy to get a global control of $E$
in the space $L^2(\R_+;\dot H^1(\R^2)),$ 
and, finally on $j$ on $L^2(\R_+;\dot H^1(\R^2))$
provided $B_0$ is small in $H^1(\R^2).$ 
We then get a global well-posedness statement, assuming only 
that the magnetic field is small.  

\medbreak
We end this introductory part presenting  a few notations. 
As usual,  we denote by $C$ harmless  positive constants that may change from line to line, and  $A\lesssim B$ means $A\leq C B.$ For $X$ a Banach space, $p\in[1, \infty]$ and 
$T>0$, the notation $L^p(0, T; X)$ or $L^p_T(X)$ designates the set of measurable functions $f: [0, T]\to X$ with $t\mapsto\|f(t)\|_X$ in $L^p(0, T)$, endowed with the norm $\|\cdot\|_{L^p_{T}(X)} :=\|\|\cdot\|_X\|_{L^p(0, T)},$ and agree that $\mathcal C([0, T];X)$ denotes the set of continuous functions from $[0, T]$ to $X$.  Sometimes, we will use the notation $L^p(X)$ to designate the space
$L^p(\R_+;X)$ and $\|\cdot\|_{L^p(X)}$ for the associated norm. 
 We will keep the same notations for multi-component functions, namely  for
 $f:[0,T]\to X^m$ with $m\in\N.$


\section{Main results}
Our first result, that has to be compared with the Fujita-Kato theorem for the incompressible Navier-Stokes equations
states that the Hall-MHD system is indeed globally well-posed 
if  $u_0,$ $B_0$ and $v_0$ are small in $\dot H^{\frac12}(\R^3).$ 
\begin{theorem}\label{Th_1} Assume that $\mu=\nu.$
Let $(u_0, B_0)\in{\dot H^{\frac{1}{2}}}(\R^3)$ with $\div u_{0}=\div B_{0}=0,$
and  $J_0:=\nabla\times B_0\in{\dot H^{\frac{1}{2}}}(\R^3).$ 
There exists a constant $c>0$  such that, if 
\begin{equation}\label{smallcon}
\|u_0\|_{\dot{H}^\frac{1}{2}(\mathbb{R}^3)}+\|B_0\|_{\dot{H}^\frac{1}{2}(\mathbb{R}^3)}+\|u_0-\eps J_0\|_{\dot{H}^\frac{1}{2}(\mathbb{R}^3)} < c\mu,
\end{equation}
 then there exists a unique global solution $$(u, B)\in\mathcal C_b(\mathbb R_+; \dot H^{\frac{1}{2}}(\R^3))\cap L^2(\mathbb R_+; \dot H^{\frac{3}{2}}(\R^3))$$ to the Cauchy problem \eqref{1.1}-\eqref{1.4}, such that $J:=\nabla\times B\in L^\infty(\mathbb R_+; \dot H^{\frac{1}{2}}(\R^3))\cap L^2(\mathbb R_+; \dot H^{\frac{3}{2}}(\R^3)).$ Furthermore,  for all $t\geq t_0\geq0,$ one has
 \begin{equation}\label{2.888}
\|(u,B,u-\eps J)(t)\|_{\dot{H}^\frac{1}{2}}^2+\frac{\mu}{2}\int_{t_0}^t \|(u,B,u-\eps J)\|_{\dot{H}^\frac{3}{2}}^2\,d\tau
\leq \|(u,B,u-\eps J)(t_0)\|_{\dot{H}^\frac{1}{2}}^2.
\end{equation}
 In particular, the function
 \begin{equation*}
 t\mapsto\|u(t)\|_{\dot{H}^\frac{1}{2}}^2+\|B(t)\|_{\dot{H}^\frac{1}{2}}^2+\|u(t)-\eps J(t)\|_{\dot{H}^\frac{1}{2}}^2\end{equation*}
 is nonincreasing.
 \medbreak
 If, in addition, $u_0\in H^s$ and  $B_0\in H^{s+1}$ for some $s\geq0$ with $s\not=1/2,$ then
$(u,B)\in\cC_b(\R_+;H^s\times H^{s+1}),$ $\nabla u\in L^2(\R_+;H^s)$ 
and~$\nabla B\in L^2(\R_+;H^{s+1}).$
 \end{theorem} 
 \begin{remark} The first part of the above 
 theorem has been proved in \cite{Da19} by a different method
 that does not allow  to get  \eqref{2.888}. 
 A small variation on the proof  yields  local well-posedness  if assuming only 
that  $\|u_0-\eps J_0\|_{\dot H^{\frac12}(\R^3)}$ is small. In contrast with the incompressible Navier-Stokes equations however, whether a local-in-time result may be proved without any smallness condition 
 (or more regularity,  see e.g. \cite[Th. 2.2]{Da19}) is an open question.\end{remark}


\begin{corollary}\label{C_1}
If $(u, B)$ denotes the solution given by Theorem \ref{Th_1} and if, in addition, $(u_0, B_0)$ is in $L^2(\mathbb{R}^3)$, then $(u, B)$ is continuous with values in $L^2(\mathbb{R}^3),$  satisfies the following energy balance for all $t\geq0:$ 
\begin{equation}\label{eq:energy}
\|u(t)\|_{L^2}^2+\|B(t)\|_{L^2}^2 
+2\mu\int_0^t\bigl(\|\nabla u\|_{L^2}^2+\|\nabla B\|_{L^2}^2\bigr)\,d\tau=
\|u_0\|_{L^2}^2+\|B_0\|_{L^2}^2,
\end{equation} 
and we have
 \begin{equation}\label{eq:decay} 
 \lim_{t\to+\infty} \bigl(\|u(t)\|_{\dot{H}^\frac{1}{2}}^2+\|B(t)\|_{\dot{H}^\frac{1}{2}}^2+\|J(t)\|_{\dot{H}^\frac{1}{2}}^2\bigr)=0.
 \end{equation}
\end{corollary}
The following corollary states that  global solutions (even if  large and with infinite energy) satisfying a suitable integrability 
property have to decay to $0$ at infinity.
\begin{corollary}\label{C_2}   Assume that  $(u_0, B_0)\in{\dot H^{\frac{1}{2}}}(\R^3)$ with $\div u_{0}=\div B_{0}=0$
and  $J_0\in{\dot H^{\frac{1}{2}}}(\R^3).$ 
Suppose in addition that  the Hall-MHD system
with $\mu=\nu$ supplemented with initial data $(u_0,B_0)$ admits a global solution  $(u,B)$ 
 such that
 $$(u, B,\nabla\times B)\in L^4(\mathbb R_+; \dot H^{1}(\R^3)).$$ 
  Then, $(u,B)$ has the regularity properties of Theorem \ref{Th_1}, and   
  \eqref{eq:decay} is satisfied.
  
  In particular, all the solutions constructed in Theorem \ref{Th_1} satisfy \eqref{eq:decay}. 
  \end{corollary}

The next theorem states a weak-strong uniqueness  property of the solution. It is valid
for all positive coefficients $\mu,$ $\nu$ and $\eps.$
\begin{theorem}\label{Th_3}
Consider initial data  $(u_{0}, B_{0})$ in $L^2(\R^3)$ with $\div u_0=\div B_0=0.$
Let $(u, B)$ be any Leray-Hopf solution of the Hall-MHD system associated with initial data $(u_{0}, B_{0}).$ Assume in addition that $u$ and $J:=\nabla\times B$ are in $L^4(0, T; \dot{H}^{1}(\mathbb{R}^3))$ for some time $T>0.$ Then, all Leray-Hopf solutions associated with $(u_0, B_0)$ coincide with $(u, B)$ on the time interval $[0, T].$
\end{theorem}
Our last result states the existence of global strong solutions for the  $2\frac12$D Hall-MHD system.
Two cases are considered : either the data are small and have just critical regularity, 
or the velocity field is more regular and only the magnetic field 
has to be small accordingly.
\begin{theorem}\label{Th_4}
Assume that $\mu=\nu.$ Let $(u_0,B_0)$ be divergence free vector-fields 
with $u_0$ in $L^2(\R^2)$ and $B_0$ in $H^1(\R^2).$ 
There exists a constant $c>0$ such that, if 
\begin{equation}\label{smallconUB}
\|u_0\|_{L^2(\R^2)} + \|B_0\|_{L^2(\R^2)} + \|u_0-\eps \wt\nabla\times B_0\|_{L^2(\mathbb{R}^2)} < c\mu,\end{equation}
then there exists a unique global solution $(u,B)$  to the Cauchy problem \eqref{1.1a}-\eqref{1.4a}, with
$(u,B)\in\mathcal{C}_b(\R_+; L^2(\R^2))\cap   L^2(\R_+; \dot H^1(\R^2))$ and 
$$\nabla B \in L^\infty(\R_+; L^2(\R^2))\with \nabla^2 B\in  L^2(\R_+; L^2(\R^2)).$$ 
If both $u_0$ and $B_0$ are in $H^1(\R^2),$ then there exists a constant $C_0$
depending only on the $L^2$ norm of $u_0,$ $\nabla u_0,$ and on $\mu,\eps$ such that if 
\begin{equation}\label{smallconB}
\|B_0\|_{H^1(\mathbb{R}^2)}\leq C_0,
\end{equation}
then there exists a unique global solution $(u,B)$  to \eqref{1.1a}-\eqref{1.4a}, with
$$(u, B)\in\mathcal{C}_b(\R_+; H^1(\R^2))\andf (\nabla u,\nabla B)\in  L^2(\R_+; {H}^1(\R^2)).$$ 
Moreover, in the two cases, \eqref{energyineq} becomes an equality.
\end{theorem}
The next section is devoted to the proof of Theorem \ref{Th_1}  and of its two corollaries. 
In section  \ref{s:weakstrong}, we establish the weak-strong uniqueness result. 
Section \ref{s:21/2} is dedicated to the proof of  Theorem \ref{Th_4}. 
A few definitions and technical results are recalled in Appendix.


\section{Small data global existence in critical Sobolev spaces}\label{s:trois}

Throughout this section, we shall 
assume for simplicity that 
$\mu=\nu=\eps=1$ (the general case $\mu=\nu>0$ and $\eps>0$ may be 
deduced after suitable rescaling). 
  We shall use repeatedly  the fact that, 
as $B$ is divergence free, one has the following
equivalence of norms for any $s\in\mathbb{R}$: 
\begin{equation}\label{eq:equivnorm}
\|\nabla B\|_{\dot{H}^{s}}\sim \|J\|_{\dot{H}^{s}},\end{equation}
and also that we have $B=\curl^{-1}(u-v),$
where the $-1$-th order homogeneous Fourier multiplier ${\rm{curl}}^{-1}$ is defined  on 
the Fourier side by 
\begin{equation}\label{eq:curl-1}\mathcal{F}({\rm{curl}}^{-1}J)(\xi):=\frac{i\xi\times\widehat{J}}{|\xi|^2}\cdotp\end{equation}

Proving the existence  part of  Theorem \ref{Th_1} is based on the following result~:
\begin{proposition}\label{p:sob}
Let $(u, B)$ be a smooth solution of  the 3D Hall-MHD system with $\eps=\mu=\nu=1,$ on the
time interval $[0,T].$  Let $v:=u-\nabla\times B.$  There exists a universal constant $C$ such that
on $[0,T],$ we have
\begin{multline}
\frac{d}{dt}(\|u\|_{\dot{H}^\frac{1}{2}}^2+\|B\|_{\dot{H}^\frac{1}{2}}^2+\|v\|_{\dot{H}^\frac{1}{2}}^2)+(\|u\|_{\dot{H}^\frac{3}{2}}^2+\|B\|_{\dot{H}^\frac{3}{2}}^2+\|v\|_{\dot{H}^\frac{3}{2}}^2)\\
\leq C\sqrt{\|u\|_{\dot{H}^\frac{1}{2}}^2+\|B\|_{\dot{H}^\frac{1}{2}}^2
+\|v\|_{\dot{H}^\frac{1}{2}}^2}\:(\|u\|_{\dot{H}^\frac{3}{2}}^2+\|B\|_{\dot{H}^\frac{3}{2}}^2+\|v\|_{\dot{H}^\frac{3}{2}}^2)\label{6.100}.
\end{multline}
\end{proposition}
\begin{proof} Applying  operator $\Lambda^{\frac{1}{2}}$ to both sides of System \eqref{1.100} 
and taking the $L^2$ scalar product with $\Lambda^{\frac{1}{2}}u$, $\Lambda^{\frac{1}{2}}B$ 
and $\Lambda^{\frac{1}{2}}v,$ respectively, we get
$$\begin{aligned}
\frac{1}{2}\frac{d}{dt}\|u\|_{\dot{H}^\frac{1}{2}}^2+\|u\|_{\dot{H}^\frac{3}{2}}^2&=(\Lambda^{\frac{1}{2}}(u\otimes u)\,|\,\nabla\Lambda^{\frac{1}{2}}u)-(\Lambda^{\frac{1}{2}}(B\otimes B)\,|\, \nabla\Lambda^{\frac{1}{2}}u)=A_1+A_2,\\
\frac{1}{2}\frac{d}{dt}\|B\|_{\dot{H}^\frac{1}{2}}^2+\|B\|_{\dot{H}^\frac{3}{2}}^2&=(\Lambda^{\frac{1}{2}}(v\times B)\,|\, \nabla\times\Lambda^{\frac{1}{2}}B)=A_3,\\
\frac{1}{2}\frac{d}{dt}\|v\|_{\dot{H}^\frac{1}{2}}^2+\|v\|_{\dot{H}^\frac{3}{2}}^2&=A_4+A_5+\cdots+A_8,
\end{aligned}$$
where
\begin{align*}
&A_4:=(\Lambda^{\frac{1}{2}}(u\otimes u)\,|\, \nabla\Lambda^{\frac{1}{2}}v),\\
&A_5:=-(\Lambda^{\frac{1}{2}}(B\otimes B)\,|\, \nabla\Lambda^{\frac{1}{2}}v),\\
&A_6:=-(\Lambda^{\frac{1}{2}}((\nabla\times v)\times B)-(\La^\frac12\nabla\times v)\times B\,|\, \Lambda^{\frac{1}{2}}\nabla\times v),\\
&A_7:=(\Lambda^{\frac{1}{2}}(v\times u)\,|\, \nabla\times\Lambda^{\frac{1}{2}}v),\\
&A_8:=2(\Lambda^{\frac{1}{2}}(v\cdot\nabla B)\,|\, \nabla\times\Lambda^{\frac{1}{2}}v).
\end{align*}
By Lemma \ref{Le_28} and  Sobolev embedding \eqref{em}, we get
\begin{align*}
|A_1|&\leq C \|\Lambda^{\frac{1}{2}}u\|_{L^3}\|u\|_{L^6}\|\nabla\Lambda^{\frac{1}{2}}u\|_{L^2}\\
&\leq C\|u\|_{\dot{H}^1}^2\|u\|_{\dot{H}^\frac{3}{2}},\\
|A_6|&\leq C(\|\nabla B\|_{L^6}\|\Lambda^{\frac{1}{2}}v\|_{L^3}+\|\Lambda^{\frac{1}{2}}B\|_{L^6}\|\nabla\times v\|_{L^3})\|v\|_{\dot{H}^\frac{3}{2}}\\
&\leq C \bigl(\|\nabla B\|_{\dot H^1}\|v\|_{\dot H^1}
+\|\nabla B\|_{\dot H^{\frac12}}\|v\|_{\dot H^{\frac32}}\bigr)\|v\|_{\dot{H}^\frac{3}{2}},\\
|A_8|&\leq C(\|\Lambda^{\frac{1}{2}}v\|_{L^3}\|\nabla B\|_{L^6}+\|v\|_{L^6}\|\nabla\Lambda^{\frac{1}{2}}B\|_{L^3})\|v\|_{\dot{H}^\frac{3}{2}}\\
&\leq C\|\nabla B\|_{\dot{H}^1}\|v\|_{\dot{H}^1}\|v\|_{\dot{H}^\frac{3}{2}}.
\end{align*}
Terms $A_2,$ $A_3,$ $A_4,$ $A_5$ and $A_7$ may be bounded similarly as $A_1$:
$$\begin{aligned}
|A_2|&\leq  C\|B\|_{\dot{H}^1}^2\|u\|_{\dot{H}^\frac{3}{2}},\\
|A_3|&\leq C\|v\|_{\dot{H}^1}\|B\|_{\dot{H}^1}\|B\|_{\dot{H}^\frac{3}{2}},\\
|A_4|&\leq C\|u\|_{\dot{H}^1}^2\|v\|_{\dot{H}^\frac{3}{2}},\\
|A_5|&\leq C\|B\|_{\dot{H}^1}^2\|v\|_{\dot{H}^\frac{3}{2}},\\
|A_7|&\leq C\|v\|_{\dot{H}^1}\|u\|_{\dot{H}^1}\|v\|_{\dot{H}^\frac{3}{2}}.
\end{aligned}$$
Hence, using repeatedly the fact that 
$$
\|z\|_{\dot H^1}\leq \sqrt{\|z\|_{\dot H^{\frac12}} \|z\|_{\dot H^{\frac32}}}
$$
and Young inequality and, sometimes, \eqref{eq:equivnorm}, it is easy to deduce \eqref{6.100}
from the above  inequalities.
\end{proof}

Now, combining Proposition \ref{p:sob} with  Lemma \ref{Le_7.000} (take $\alpha=1$, $W\equiv0$) implies that 
there exists a constant $c>0$ such that    if 
\begin{equation}\label{eq:small}\|u_0\|_{\dot{H}^\frac{1}{2}}+\|B_0\|_{\dot{H}^\frac{1}{2}}+\|v_0\|_{\dot{H}^\frac{1}{2}}< c,\end{equation}
then we have for all time $t\geq 0,$
\begin{equation}\label{6.888}
\|(u,B,v)(t)\|_{\dot{H}^\frac{1}{2}}^2+\frac12\int_0^t \|(u,B,v)\|_{\dot{H}^\frac{3}{2}}^2\,d\tau
\leq \|(u,B,v)(0)\|_{\dot{H}^\frac{1}{2}}^2.
\end{equation}
That inequality obviously implies that Condition \eqref{eq:small} is 
satisfied for all time $t_0.$ 
Hence, repeating the argument, we get  \eqref{2.888} for all $t\geq t_0\geq0.$ 
\medbreak
In order to prove rigorously 
the existence part of Theorem \ref{Th_1}, one can resort
to  the following classical procedure:
\begin{enumerate}
\item   smooth out the initial data and get a sequence  $(u^n, B^n)_{n\in\N}$
 of smooth solutions to Hall-MHD system on the maximal time interval $[0, T^n)$;
\item  apply \eqref{6.888} to $(u^n, B^n)_{n\in\N}$ and prove that $T^n=\infty$ and   that  
the sequence $(u^n, B^n, v^n)_{n\in\N}$ with $v^n:=u^n-\nabla\times B^n$
is bounded in  $L^\infty(\R_+; \dot{H}^\frac12)\cap L^2(\R_+; \dot{H}^\frac32)$;
\item  use compactness to prove that  $(u^n, B^n)_{n\in\N}$ converges, up to extraction, 
to a solution of Hall-MHD
system supplemented with initial data $(u_0,B_0)$;
\item prove stability estimates in $L^2$ to get the uniqueness of the solution.
\end{enumerate}
\medbreak
To proceed, let us  smooth out the initial data as follows:
\begin{equation*}
u^n_0 :=\mathcal{F}^{-1}(\mathbf{1}_{\cC_n}\widehat{u_0})\andf  B^n_0 :=\mathcal{F}^{-1}(\mathbf{1}_{\cC_n}\widehat{B_0}),
\end{equation*}
where $\cC_n$ stands for the annulus 
with small radius $n^{-1}$ and large radius $n.$ 
Clearly, $u^n_0$ and $B^n_0$ belong to all Sobolev spaces, and
\begin{equation}\label{smoothdata}
\|(u^n_0, B^n_0, v^n_0)\|_{\dot{H}^\frac12}\leq \|(u_0, B_0, v_0)\|_{\dot{H}^\frac12}\cdotp
\end{equation} 
The classical well-posedness theory in Sobolev spaces (see e.g. \cite{Ch13})  ensures that 
the Hall-MHD system with data $(u_0^n,B_0^n)$  has a unique maximal solution $(u^n,B^n)$ 
on $[0,T^n)$ for some $T^n>0,$  belonging to $\cC([0, T]; H^m)\cap L^2(0, T; H^{m+1})$ for all $T<T^n.$ 
Since the solution is smooth, we have according to \eqref{6.888} and \eqref{smoothdata}, 
$$\|(u^n, B^n, v^n)\|_{L^\infty_{T^n}(\dot{H}^\frac12)}^2+\frac12\|(u^n, B^n, v^n)\|_{L^2_{T^n}(\dot{H}^\frac32)}^2\leq \|(u_0, B_0, v_0)\|_{\dot{H}^\frac12}^2\cdotp$$
\medbreak
Now, using \eqref{eq:equivnorm} and the embedding  $\dot{H}^\frac32(\R^3)\hookrightarrow BMO(\R^3)$, we get
$$\int_0^{T^n} \|(u^n, \nabla B^n)\|^2_{BMO}\,dt
\lesssim \|(u^n, v^n)\|_{L^2_{T^n}(\dot{H}^\frac32)}^2.$$
Hence, the continuation criterion of \cite{Ch14} guarantees that  $T^n=+\infty.$
This means that the solution is global and that, furthermore, 
\begin{equation}\label{6.220}
\|(u^n, B^n, v^n)\|_{L^\infty(\dot{H}^\frac12)}^2+\frac12\|(u^n, B^n, v^n)\|_{L^2(\dot{H}^\frac32)}^2\leq \|(u_0, B_0, v_0)\|_{\dot{H}^\frac12}^2\cdotp
\end{equation} 
We claim that, up to  extraction, the sequence $(u^n, B^n)_{n\in\mathbb{N}}$ converges in $\mathcal{D}'(\mathbb{R}_+\times\mathbb{R}^3)$ to a solution $(u, B)$ of \eqref{1.1}--\eqref{1.3} supplemented
with data $(u_0,B_0).$ The definition of $(u^n_0, B^n_0)$ and the fact that $(u_0,B_0,v_0)$ belongs to $\dot H^{\frac12}$  already entails that 
\begin{equation*}
(u^n_0, B^n_0, v^n_0)\to(u_0, B_0, v_0)\quad{\rm{in}}\quad\dot{H}^{\frac12}\cdotp
\end{equation*}
\medbreak
Proving  the convergence of $(u^n, B^n, v^n)_{n\in\N}$ can be achieved 
from  compactness arguments, after  exhibiting  bounds in suitable spaces
for $(\partial_tu^n,\partial_t B^n, \partial_tv^n)_{n\in\N}.$ Then, combining
with compact embedding will enable us to  apply Ascoli's theorem and to get the existence of a limit $(u, B, v)$ for a subsequence. Furthermore, the uniform bound \eqref{6.220}  will  provide us with additional regularity and convergence properties so that we will be able to pass to the limit in 
the Hall-MHD system.
   
   To proceed, let us introduce  $(u_L,B_L,v_L):=e^{t\Delta}(u_0,B_0,v_0),$ 
    $(u^n_L, B^n_L, v^n_L):=\mathcal{F}^{-1}\bigl(\mathbf{1}_{\cC_n}(\wh u_L,\wh B_L,\wh v_L)\bigr)$
    and  $(\wt u^n, \wt B^n, \wt v^n):=(u^n-u^n_L, B^n-B^n_L, v^n-v^n_L).$ 
 
It is clear that $(u^n_L, B^n_L, v^n_L)$ tends to $(u_L,B_L,v_L)$ in $\mathcal{C}(\R_+; \dot{H}^\frac12)\cap L^2(\R_+; \dot{H}^\frac32),$
which implies that $v_L=u_L-\nabla\times B_L,$ since  $v_L^n=u_L^n-\nabla\times B_L^n$
for all $n\in\N.$

Proving the convergence of $(\wt u^n,\wt B^n,\wt v^n)_{n\in\N}$  relies on  the following lemma:
\begin{lemma}\label{Le_6.4} Sequence $(\wt u^n, \wt B^n, \wt v^n)_{n\in\mathbb{N}}$ is   bounded in $\cC^{\frac14}_{loc}(\R_+; \dot{H}^{-1}).$
\end{lemma}
\begin{proof}
Observe  that $(\wt u^n, \wt B^n, \wt v^n)(0)=(0, 0, 0)$ and that
\begin{equation}\label{AP}
\left\{\begin{aligned}
\partial_t \wt u^n&=\Delta \wt u^n+\mathcal{P}\bigr(\div(B^n\otimes B^n)-\div(u^n\otimes u^n)\bigl),\\
\partial_t \wt B^n&=\Delta \wt B^n+\nabla\times(v^n\times B^n),\\
\partial_t \wt v^n&=\Delta \wt v^n+\mathcal{P}\Bigr(\div(B^n\otimes B^n)-\div(u^n\otimes u^n)\Bigr)
\\&-\nabla\times((\nabla\times v^n)\times B^n)
+\nabla\times(v^n\times u^n)+2\nabla\times(v^n\cdot\nabla B^n).
\end{aligned}
\right.
\end{equation}
Using   the uniform bound \eqref{6.220}, the  product laws:
$$\|a\,b\|_{L^2}\lesssim\|a\|_{\dot{H}^{\frac12}}\|b\|_{\dot{H}^{1}}\andf
\|(\nabla\times a)\times b\|_{L^2}\lesssim\|a\|_{\dot{H}^{1}}\|b\|_{L^\infty},$$
and  Gagliardo-Nirenberg inequality \eqref{eq:GN2} with $s=1$ and $s'=2,$
 we discover
  that the right-hand side of \eqref{AP} is uniformly bounded in $L^{\frac43}_{loc}(\R_+;\dot{H}^{-1}).$ 
Then applying  H\"older
 inequality completes the proof of the lemma. \end{proof}
\medbreak
One can now turn to the proof of the existence of a solution. Let $(\phi_{j})_{j\in\mathbb{N}}$ be a sequence of $\cC^\infty_c(\mathbb{R}^3)$ cut-off functions supported in the ball $B(0, j+1)$ of $\mathbb{R}^3$ and equal to $1$ in a neighborhood of $B(0, j).$

 Lemma \ref{Le_6.4} ensures that $(\wt u^n, \wt B^n, \wt v^n)_{n\in\mathbb{N}}$ is uniformly equicontinuous in the space $\cC([0,T]; \dot{H}^{-1})$ 
 for all $T>0,$ and \eqref{6.220} tells  us that   it is bounded in $L^\infty(\R_+;\dot H^{\frac12}).$ 
  Using the fact that the application $z\mapsto\phi_{j}z$ is compact from $\dot{H}^{\frac12}$ into 
  $\dot{H}^{-1},$  combining  Ascoli's theorem and Cantor's diagonal process enables to conclude that there exists some triplet $(\wt u, \wt B, \wt v)$  such that for all $j\in\mathbb{N},$ 
\begin{equation}\label{6.444}
(\phi_{j}\wt u^n, \phi_{j}\wt B^n, \phi_{j}\wt v^n)\to(\phi_{j}\wt u, \phi_{j}\wt B, \phi_{j}\wt v)\quad {\rm{in}}\quad \cC(\R_+; \dot{H}^{-1}).
\end{equation}
This obviously entails that $(\wt u^n, \wt B^n, \wt v^n)$ tends to $(\wt u, \wt B, \wt v)$ in $\cD'(\mathbb{R}_+\times\mathbb{R}^3),$
which is enough to pass to the limit in all the linear terms of  \eqref{1.100}
and to ensure that $\wt v=\wt u-\nabla\times\wt B,$ and thus $v=u-\nabla\times B.$

{}From the  estimates \eqref{6.220}, interpolation and  classical  functional analysis arguments,
 we  gather that $(\wt u, \wt B, \wt v)$ belongs to $L^\infty(0,T;\dot H^{\frac12})\cap L^2(0,T;\dot H^{\frac32})$ and to  $\cC^{\frac14}([0, T]; \dot{H}^{-1})$ for all $T>0,$ 
 and  better properties of convergence like, for instance,
 \begin{equation}\label{eq:better}
 \phi_j(\wt u^n, \wt B^n, \wt v^n)\to \phi_j(\wt u,\wt B,\wt v)\quad\hbox{in }\ 
 L^2_{loc}(\R_+;\dot H^1)\quad\hbox{for all }\ j\in\N.
 \end{equation} 
 As an example, let us  explain how to  pass to the limit in the term $\nabla\times((\nabla\times v^n)\times B^n).$ 
 Let $\theta\in\cC^\infty_c(\mathbb{R}_+\times\mathbb{R}^3;\R^3)$ and $j\in\mathbb{N}$ be such that ${\rm{Supp}}\,\theta\subset[0,\,j]\times B(0, j).$ 
We write
$$\displaylines{
\langle\nabla\times((\nabla\times v^n)\times B^n),\,\theta\rangle-\langle\nabla\times((\nabla\times v)\times B), \theta\rangle\hfill\cr\hfill
=\langle(\nabla\times v^n)\times \phi_j(B^n-B),\,\nabla\times\theta\rangle
+\langle(\nabla\times\phi_j(v^n-v))\times B,\,\nabla\times\theta\rangle.}$$
Now, we have for all $T>0,$ 
\begin{align*}
\|(\nabla\times v^n)\times \phi_j(B^n-B)\|_{L^1(0,T; L^2)}\lesssim \|\nabla\times v^n\|_{L^2(0,T;\dot{H}^\frac12)}\|\phi_j(B^n-B)\|_{L^2(0,T;\dot{H}^1)},
\end{align*}
\begin{align*}
\|(\nabla\times\phi_j(v^n-v))\times B\|_{L^\frac43(0,T; L^2)}\lesssim\|(\nabla\times\phi_j(v^n-v))\|_{L^2(0,T\times \R^3)}\|B\|_{L^4(0,T;L^\infty)}.
\end{align*}
Thanks to \eqref{6.220} and to \eqref{eq:better}, we  see that the right-hand sides above converge to $0.$
Hence
$$\nabla\times((\nabla\times v^n)\times B^n)\to\langle\nabla\times((\nabla\times v)\times B)\quad\hbox{in }\ \cD'(\R_+\times\R^3).$$
Arguing similarly to pass to the limit in the other nonlinear terms, one may 
conclude that $(u,B,v)$ satisfies the extended formulation \eqref{1.100}. 
Besides, as we know that  $v=u-\nabla\times B,$ the couple  $(u,B)$ satisfies the Hall-MHD system 
for some suitable pressure function $P.$ 
\medbreak
To prove that $(u, B)$ is continuous in $\dot{H}^{\frac12},$ it suffices  to notice that
the properties of regularity of the solution  ensure that $u$ and $B$ satisfy a heat equation with 
initial data in $\dot H^{\frac12}$ and right-hand side in $L^2(\R_+;\dot H^{-\frac12})$
(we do not know how to prove the time continuity  with values  in $\dot{H}^{\frac12}$
for $\nabla B$ or, equivalently, $v,$  though). 
\medbreak
Let us next prove the  uniqueness part of the theorem. Let   $(u_1, B_1)$ and $(u_2, B_2)$ be two solutions 
of the Hall-MHD system on $[0,T]\times\R^3,$ supplemented with the same initial data $(u_0, B_0)$ and such that, denoting $v_i=u_i-\nabla\times B_i$ for  $i=1,2,$ 
$$(u_i,B_i,v_i)\in L^\infty([0,T];\dot H^{\frac12})\andf 
(\nabla u_i,\nabla B_i,\nabla v_i)\in L^2(0,T;\dot H^{\frac12}).$$
In order to prove the result, we shall estimate the difference  $(\du, \dB, \dv) := (u_1-u_2, B_1-B_2, v_1-v_2)$ in  the space $\cC([0,T]; L^2)\cap 
L^2(0,T; \dot H^1).$  
In order to  justify  that, indeed, $(\du, \dB, \dv)$  belongs to that space, 
one can observe that 
\begin{equation}\label{3.d1}
\left\{
\begin{aligned}
 &\partial_t {\du}-\Delta {\du}:=R_1,\\
 &\partial_t {\dB}-\Delta {\dB}:=R_2,\\
 &\partial_t \dv-\Delta\dv:=R_1+R_3+R_4+ R_5,
\end{aligned}
\right.
\end{equation}
where
\begin{align*}
&R_1:=\mathcal{P}(B_1\cdot\nabla\dB+\dB\cdot\nabla B_2-u^1\cdot\nabla\du-\du\cdot\nabla u_2),\\
&R_2:=\nabla\times(v_1\times\dB+\dv\times B_2),\\
&R_3:=-\nabla\times((\nabla\times v_1)\times\dB+(\nabla\times\dv)\times B_2),\\
&R_4:=\nabla\times(v_1\times\du+\dv\times u_2),\\
&R_5:=2\nabla\times(v_1\cdot\nabla\dB+\dv\cdot\nabla B_2).
\end{align*}
Since  $(\du, \dB, \dv)|_{t=0}{=}0,$  in order to achieve our goal, it suffices to prove that 
$R_1$ to $R_5$ belong to the space $L^2(0,T;\dot H^{-1}).$
Now, since  $(\du, \dB, \dv)\in L^\infty(0,T; \dot H^{\frac{1}{2}})\cap L^2(0,T; \dot H^{\frac{3}{2}}),$   we have, by interpolation and H\"older inequality 
that   $(\du, \dB, \dv)\in  L^4(0,T; \dot H^{1}).$
Hence, using repeatedly the fact that the numerical product of functions
may be continuously extended to $\dot H^1\times \dot H^{1/2}\to L^2,$ 
one can write that 
$$\begin{aligned}
\|R_1\|_{L^2(0,T;\dot{H}^{-1})}
&\lesssim\|B_1\otimes\dB\|_{L^2(0,T;L^2)}+\|B_2\otimes\dB\|_{L^2(0,T;L^2)}\\
&\hspace{3.4cm}+\|u_1\otimes\du\|_{L^2(0,T;L^2)}+\|u_2\otimes\du\|_{L^2(0,T;L^2)}\\
&\lesssim T^{\frac14}\|(u_1, u_2, B_1, B_2)\|_{L^{\infty}(0,T;\dot H^{\frac{1}{2}})}\|(\du, \dB)\|_{L^{4}(0,T;\dot H^{1})},\\
\|R_2\|_{L^2(0,T;\dot{H}^{-1})}
&\lesssim\|v^1\times\dB\|_{L^2(0,T;L^2)}+\|\dv\times B_2\|_{L^2(0,T;L^2)}\\
&\lesssim T^{\frac14}\|(v_1, B_2)\|_{L^{\infty}(0,T;\dot H^{\frac{1}{2}})}\|(\dv, \dB)\|_{L^{4}(0,T;\dot H^{1})},
\end{aligned}
$$
$$
\begin{aligned}\|R_3\|_{L^2(0,T;\dot{H}^{-1})}
&\lesssim\|(\nabla\times v_1)\times\dB\|_{L^2(0,T;L^2)}+\|(\nabla\times \dv)\times B_2\|_{L^2(0,T;L^2)}\\
&\lesssim\|v_1\|_{L^2(0,T;\dot H^{\frac32})}\|\dB\|_{L^\infty(0,T;\dot H^{1})}+\|\dv\|_{L^2(0,T;\dot H^{\frac32})}\|B_2\|_{L^\infty(0,T;\dot H^{1})}\\
&\lesssim\|(v_1, \dv)\|_{L^2(0,T;\dot H^{\frac32})}\|(B_2, \dB)\|_{L^\infty(0,T;\dot H^{1})},\end{aligned}$$
Note that our assumptions ensure that $B_i$ and $\nabla B_i$ are in $L^\infty(0,T;\dot H^{\frac12})$
and thus we do have, by interpolation inequality \eqref{i1}, $B_i$ in $L^\infty(0,T; \dot H^{1})$ for $i=1,2.$
Terms $R_4$ and $R_5$ may be treated similarly. 
\medbreak
Estimating   $(\du, \dB, \dv)$ in  the space $L^\infty(0,T; L^2)\cap 
L^2(0,T; \dot H^1)$ follows from  a standard energy method applied on \eqref{3.d1}, H\"older's inequality and Sobolev embedding.
More precisely, we have
$$\begin{aligned}
\frac12\frac{d}{dt}\|\du\|_{L^2}^2\!+\!\|\du\|_{\dot H^1}^2
&\!\lesssim\!(\|B_1\otimes\dB\|_{L^2}\!+\!\|B_2\otimes\dB\|_{L^2}+\|u_2\otimes\du\|_{L^2})\|\nabla\du\|_{L^2}\\
&\!\lesssim\! (\|B_1\|_{\dot{H}^{1}}\|\dB\|_{\dot{H}^{\frac{1}{2}}}\!+\!\|B_2\|_{\dot{H}^{1}}\|\dB\|_{\dot{H}^{\frac12}}\!+\!\|u_2\|_{\dot{H}^{1}}\|\du\|_{\dot{H}^{\frac12}})\|\du\|_{\dot{H}^{1}},
\end{aligned}$$
$$\begin{aligned}
\frac12\frac{d}{dt}\|\dB\|_{L^2}^2\!+\!\|\dB\|_{\dot H^1}^2
\lesssim&(\|v_1\times\dB\|_{L^2}+\|\dv\times B_2\|_{L^2})\|\nabla\dB\|_{L^2}\\
\lesssim&(\|v_1\|_{\dot{H}^{1}}\|\dB\|_{\dot{H}^{\frac{1}{2}}}+\|\dv\|_{\dot{H}^{\frac12}}\|B_2\|_{\dot{H}^{1}})\|\dB\|_{\dot{H}^1}
\end{aligned}$$
and, using \eqref{eq:equivnorm}, 
$$\begin{aligned}
\frac12\frac{d}{dt}\|\dv&\|_{L^2}^2+\|\dv\|_{\dot H^1}^2\\
&\lesssim(\|B_1\otimes\dB\|_{L^2}+\|B_2\otimes\dB\|_{L^2}+\|u_1\otimes\du\|_{L^2}+\|u_2\otimes\du\|_{L^2})\|\nabla\dv\|_{L^2}\\
&\qquad\qquad+\|(\nabla\times v_1)\times\dB\|_{L^2}\|\nabla\times\dv\|_{L^2}\\
&\qquad\qquad+(\|v_1\times\du\|_{L^2}+\|\dv\times u_2\|_{L^2})\|\nabla\times\dv\|_{L^2}\\
&\qquad\qquad+(\|v_1\cdot\nabla\dB\|_{L^2}+\|\dv\cdot\nabla B_2\|_{L^2})\|\nabla\times\dv\|_{L^2}\\
&\lesssim\Bigl(\|B_1\|_{\dot{H}^{1}}\|\dB\|_{\dot{H}^{\frac{1}{2}}}+\|B_2\|_{\dot{H}^{1}}\|\dB\|_{\dot{H}^{\frac12}}+\|u_1\|_{\dot{H}^{1}}\|\du\|_{\dot{H}^{\frac{1}{2}}}+\|u_2\|_{\dot{H}^{1}}\|\du\|_{\dot{H}^{\frac12}}\\
&\qquad\qquad+\|\nabla\!\times\! v_1\|_{\dot{H}^{\frac{1}{2}}}(\|\du\|_{L^2}\!+\!\|\dv\|_{L^2})
\!+\!\|v_1\|_{\dot{H}^1}\|\du\|_{\dot{H}^{\frac{1}{2}}}\!+\!\|\dv\|_{\dot{H}^1}\|u_2\|_{\dot{H}^{\frac{1}{2}}}\\
&\qquad\qquad+\|v_1\|_{\dot{H}^1}(\|\du\|_{\dot{H}^{\frac{1}{2}}}+\|\dv\|_{\dot{H}^{\frac{1}{2}}})+\|\dv\|_{\dot{H}^{\frac12}}\|\nabla B_2\|_{\dot{H}^{1}}\Bigr)\|\dv\|_{\dot{H}^1}.
\end{aligned}$$
At this stage,  interpolation and Young's inequality imply that
$$
\begin{aligned}
\|B_1\|_{\dot{H}^{1}}\|\dB\|_{\dot{H}^{\frac{1}{2}}}\|\du\|_{\dot{H}^{1}}&\lesssim
\|B_1\|_{\dot{H}^{1}}\|\dB\|_{L^2}^{\frac12}\|\dB\|_{\dot{H}^{1}}^{\frac12}\|\du\|_{\dot{H}^{1}}\\
&\leq \frac1{10} \|(\dB,\du)\|_{\dot H^1}^2 +C\|\dB\|_{L^2}^2\|B_1\|_{\dot H^1}^4,
\end{aligned}
$$
and similar inequalities for all the terms  of the right-hand sides of the above inequalities, except for 
the one with $\|\nabla \times v_1\|_{\dot H^{\frac12}}$ that we bound as follows:
$$\|\nabla\times v_1\|_{\dot H^{\frac12}}\|(\du,\dv)\|_{L^2}\|\dv\|_{\dot H^1}
\leq \frac1{10}\|\dv\|_{\dot H^1}^2 + C \|v_1\|_{\dot H^{\frac32}}^2\|(\du,\dv)\|_{L^2}^2.$$
In the end, we get for all $t\in(0, T),$
$$\frac12\frac{d}{dt}\|(\du,\dB,\dv)\|_{L^2}^2+\|(\du,\dB,\dv)\|_{\dot H^1}^2 \lesssim
V(t)\|(\du,\dB,\dv)\|_{L^2}^2 $$
with
\begin{equation*}
V(t):=\|(u_1,u_2, B_1,B_2, v_1,v_2)(t)\|_{\dot{H}^{1}}^4+\|v_1(t)\|_{\dot{H}^{\frac{3}{2}}}^2.
\end{equation*}
Since our assumptions ensure that $V$ is integrable on $[0,T]$ 
and  $(\du,\dB,\dv)(0)=0,$ applying  Gronwall's inequality yields 
$$(\du, \dB, \dv)\equiv0\ \hbox{ in }\ L^\infty(0, T; L^2(\mathbb R^{3})).$$

 Let us finally explain  the propagation of higher Sobolev regularity
if the initial data $(u_0,B_0)$ are, additionally, in $H^s\times H^{s+1}$  for some  $s\geq0.$
Our aim is to prove that the solution $(u,B)$ constructed above is in  
$\cC_b(\R_+;H^s\times H^{s+1}),$ 
and such that $(\nabla u,\nabla B)\in L^2(\R_+;H^s\times H^{s+1}).$ 

For the time being, let us assume that $(u,B)$ is smooth and explain 
how to  perform estimates in Sobolev spaces. First, we multiply \eqref{1.1} and \eqref{1.3} by $u$ and $B$, respectively, integrate and add up the resulting equations.  Using the fact that  
\begin{equation*}
(\nabla\times(J\times B)\,|\,B)=(J\times B\,|\,J)=0,
\end{equation*}
one gets  the following energy balance:
\begin{equation}
\frac{1}{2}\frac{d}{dt}(\|u\|_{L^2}^2+\|B\|_{L^2}^2)+\|\nabla u\|_{L^2}^2+\|\nabla B\|_{L^2}^2=0\label{3.1500}.
\end{equation}
Let $\Lambda^s$ denote the fractional derivative operator defined in the Appendix. Since 
$$\|u\|_{H^{s}}+ \|B\|_{H^{s+1}}\approx \|(u,B)\|_{L^2}+\|(\Lambda^s u,\Lambda^s B,\Lambda^s v)\|_{L^2},$$
in order to prove the desired Sobolev estimates, it suffices
to get a suitable control on $\|\Lambda^s u\|_{L^2}$ and on $\|\Lambda^{s+1}B\|_{L^2}.$
To this end, apply  $\Lambda^s$ to  \eqref{1.1}, then take the $L^2$ scalar product with $\Lambda^s u.$ We get: 
$$\displaylines{\quad
\frac{1}{2}\frac{d}{dt}\|\Lambda^s u\|_{L^2}^2+\|\Lambda^s\nabla u\|_{L^2}^2=(\Lambda^s(B\otimes B)\,|\,\Lambda^s \nabla u)-(\Lambda^s(u\otimes u)\,|\, \Lambda^s \nabla u)=:E_1+E_2.
}$$
In order to control $\|\Lambda^{s+1}B\|_{L^2},$  one has to use the cancellation property \eqref{1.1100}.
Then, applying  $\Lambda^s$ to the second and third equation of \eqref{1.100} and taking the  $L^2$ scalar product with $\Lambda^s B,~\Lambda^s v,$ respectively,  yields:
$$\displaylines{\quad
\frac{1}{2}\frac{d}{dt}\|\Lambda^s B\|_{L^2}^2+\|\Lambda^s\nabla B\|_{L^2}^2=(\Lambda^s(v\times B)\,|\,\Lambda^s \nabla\times B)=:E_3,
}$$
$$\displaylines{\quad
\frac{1}{2}\frac{d}{dt}\|\Lambda^s v\|_{L^2}^2+\|\Lambda^s\nabla v\|_{L^2}^2=(\Lambda^s(B\otimes B)\,|\, \Lambda^s \nabla v)-(\Lambda^s(u\otimes u)\,|\, \Lambda^s\nabla v)\hfill\cr\hfill
\hspace{2cm}-(\Lambda^s((\nabla\times v)\times B)-(\Lambda^s\nabla\times v)\times B\,|\,\Lambda^s\nabla\times v)+(\Lambda^s(v\times u)\,|\, \Lambda^s\nabla\times v)\hfill\cr\hfill+2(\Lambda^s(v\cdot\nabla B)\,|\,\Lambda^s\nabla\times v)
=:E_4+E_5+\cdots+E_8.
}$$
 Sobolev embedding, Young's inequality and Lemma \ref{Le_28}, imply that
$$\begin{aligned}
|E_1|&\lesssim\|\Lambda^s B\|_{L^6}\|B\|_{L^3}\|\Lambda^s\nabla u\|_{L^2}\\
&\lesssim\|B\|_{\dot{H}^{\frac12}}(\|\Lambda^s\nabla B\|_{L^2}^2+\|\Lambda^s\nabla u\|_{L^2}^2),\\
|E_2|
&\lesssim\|u\|_{\dot{H}^{\frac12}}\|\Lambda^s\nabla u\|_{L^2}^2,\\
|E_3|&\lesssim(\|\Lambda^s v\|_{L^6}\|B\|_{L^3}+\|\Lambda^s B\|_{L^6}\|v\|_{L^3})\|\Lambda^s\nabla\times B\|_{L^2}\\
&\lesssim\|(B, v)\|_{\dot{H}^{\frac12}}(\|\Lambda^s\nabla B\|_{L^2}^2+\|\Lambda^s\nabla v\|_{L^2}^2),\\
|E_4|
&\lesssim\|B\|_{\dot{H}^{\frac12}}(\|\La^s \nabla B\|_{L^2}^2+\|\La^s \nabla v\|_{L^2}^2),\\
|E_5|
&\lesssim\|u\|_{\dot{H}^{\frac12}}(\|\La^s \nabla u\|_{L^2}^2+\|\La^s \nabla v\|_{L^2}^2),\\
|E_6|&\lesssim(\|\nabla B\|_{L^3}\|\La^{s-1}\nabla\times v\|_{L^6}+\|\La^s B\|_{L^6}\|\nabla\times v\|_{L^3})\|\La^s \nabla v\|_{L^2}\\
&\leq (C\|\nabla B\|_{\dot{H}^{\frac12}}+{\textstyle{\frac12}})\|\La^s \nabla v\|_{L^2}^2+C\|\La^s \nabla B\|_{L^2}^2\|v\|_{\dot{H}^{\frac32}}^2,\\
|E_7|&\lesssim(\|\La^s v\|_{L^6}\|u\|_{L^3}+\|\La^s u\|_{L^6}\|v\|_{L^3})\|\La^s \nabla v\|_{L^2}\\
&\lesssim\|(u, v)\|_{\dot{H}^{\frac12}}(\|\La^s \nabla u\|_{L^2}^2+\|\La^s \nabla v\|_{L^2}^2),\\
|E_8|&\lesssim(\|\La^s v\|_{L^6}\|\nabla B\|_{L^3}+\|\La^s\nabla B\|_{L^6}\|v\|_{L^3})\|\La^s \nabla v\|_{L^2}\\
&\lesssim\|(\nabla B, v)\|_{\dot{H}^{\frac12}}(\|\La^s \nabla B\|_{\dot{H}^1}^2+\|\La^s \nabla v\|_{L^2}^2).
\end{aligned}$$
Since $\|(u, B, v)\|_{\dot{H}^\frac12}$ is small, putting the above estimates  and \eqref{3.1500} together, and using   \eqref{eq:equivnorm}, one gets after time integration that
$$\displaylines{\quad
\|u(t)\|_{\dot H^s}^2+\|B(t)\|_{\dot H^s}^2+\|v(t)\|_{\dot{H}^s}^2+\int^t_0(\|\nabla u\|_{\dot H^s}^2+\|\nabla B\|_{\dot H^s}^2+\|\nabla v\|_{\dot{H}^s}^2)\,d\tau\hfill\cr\hfill
\leq \|u_0\|_{\dot H^s}^2+\|B_0\|_{\dot H^s}^2+\|v_0\|_{\dot{H}^s}^2+C \int_0^t(\|u\|_{\dot H^s}^2+\|v\|_{\dot{H}^s}^2)\|v\|_{\dot{H}^\frac32}^2\,d\tau.\quad}$$
By Gronwall  inequality, we then get for all $t\geq 0,$
$$\displaylines{\quad
\|(u,B,v)(t)\|_{\dot H^s}^2+\int^t_0\|(\nabla u,\nabla B,\nabla v)\|_{\dot H^{s}}^2\,d\tau\hfill\cr\hfill
\leq \|(u_0,B_0,v_0)\|_{{\dot H}^{s}}\exp\biggl(C\int_0^t \|v\|_{\dot{H}^\frac32}^2\,d\tau\biggr)\cdotp\quad}
$$
Putting together with \eqref{3.1500} and using that  $\int_0^t \|v(\tau)\|_{\dot{H}^\frac32}^2\,d\tau$ is bounded thanks 
to the first part of the theorem, we get a global-in-time control of the Sobolev norms. 

Of course, to make the proof rigorous, one has to smooth out the data.  For that, one 
 can proceed exactly as in  \cite{Da19}.
 This completes the proof of Theorem \ref{Th_1}.\hfill$\square$
 \medbreak
 \begin{proof}[Proof of Corollary \ref{C_1}]
As the solution $(u, B)$ belongs to 
$$\mathcal C_b(\mathbb R_+; \dot H^{\frac{1}{2}}(\mathbb{R}^3))\cap L^2(\mathbb R_+; \dot H^{\frac{3}{2}}(\mathbb{R}^3))$$
and $J\in\cC_b(\mathbb R_+; \dot H^{\frac{1}{2}}(\mathbb{R}^3))\cap L^2(\mathbb R_+; \dot H^{\frac{3}{2}}(\mathbb{R}^3)),$ the interpolation inequality between Sobolev norms implies that $(u, B, J)$ belongs to the space $L^8_{loc}(\R_+;\dot{H}^{\frac34}(\mathbb{R}^3)),$ which, in view of Sobolev embedding, is a subspace of $L^4_{loc}(\R_+;L^4(\mathbb{R}^3)).$ Now,  the right-hand sides of the first two equations of \eqref{1.100}
belongs to $L^2_{loc}(\R_+;\dot{H}^{-1}(\mathbb{R}^3))$, which 
ensures time continuity.
\smallbreak
 In order to prove that  the energy balance is fulfilled, 
one can use the same approximation scheme as in the proof of existence
(the energy balance is clearly satisfied by $(u_n,B_n)$) then observe
that $(u_n,B_n)_{n\in\N}$ is actually a Cauchy sequence in 
$L^\infty(\R_+;L^2(\R^3))\cap L^2(\R_+;\dot H^1(\R^3)),$ as may be checked
by arguing as in the proof of uniqueness. 
\smallbreak
Let us next prove that $(u,B,v)$  goes to $0$ in $\dot H^{\frac12}(\mathbb{R}^3)$ when $t\to+\infty.$ Inequality 
\eqref{2.888} and interpolation  guarantee that $B\in L^4(\R_+;\dot H^1(\mathbb{R}^3)).$ 
Hence one can find  some $t_0\geq0$
so that $v(t_0)\in L^2(\mathbb{R}^3).$  Then, performing 
an energy argument on the equation satisfied by~$v,$ we get for all $t\geq t_0,$
$$\displaylines{\quad
\|v(t)\|_{L^2}^2+\int_{t_0}^t\|\nabla v\|_{L^2}^2\,d\tau\leq
\|v(t_0)\|_{L^2}^2\hfill\cr\hfill+\int_{t_0}^t \bigl(\|B\otimes B-u\otimes u\|_{L^2}
+\|v\times u\|_{L^2}+\|v\cdot\nabla B\|_{L^2}\bigr)^2\,d\tau.\quad}
$$
Using repeatedly the product law $\dot H^1(\mathbb{R}^3)\times\dot H^{\frac12}(\mathbb{R}^3)\to L^2(\mathbb{R}^3),$ the equivalence \eqref{eq:equivnorm} 
and adding up to \eqref{eq:energy} yields
$$\displaylines{\quad
\|(u,B,v)(t)\|_{L^2}^2+\int_{t_0}^t\|(u,B,v)\|_{\dot H^1}^2\leq
\|(u,B,v)(t_0)\|_{L^2}^2\hfill\cr\hfill+C\int_{t_0}^t\|(u,B,v)\|_{\dot H^1}^2\|(u,B,v)\|_{\dot H^{\frac12}}^2\,d\tau.\quad}
$$
Since $\|(u,B,v)(t)\|_{\dot H^{\frac12}}$ is small for all $t\geq0,$ the last term may be
absorbed by the left-hand side, and one can conclude (by interpolation)
that $(u,B,v)\in L^4(t_0,+\infty;\dot H^{\frac12}(\mathbb{R}^3)).$ 
Therefore, for all $\sigma>0$ one may find some $t_1\geq t_0$ so that 
$\|(u,B,v)(t_1)\|_{\dot H^{\frac12}}\leq\sigma.$ Combining with \eqref{2.888} allows
to conclude the proof of \eqref{eq:decay}.
\end{proof}

  \medbreak
 \begin{proof}[Proof of Corollary \ref{C_2}]
   We shall  argue as  in \cite{Ba11} and \cite{Ga03}, splitting  the 
   data into  a small part in $\dot H^{\frac12}(\R^3)$ and a (possibly) large part in 
   $H^{\frac12}(\R^3).$ More precisely, we set    
    $$v_0=u_0-\nabla\times B_0,\quad u_0=u_{0,\ell}+u_{0,h},\quad B_0=B_{0,\ell}+B_{0,h},\quad v_0=v_{0,\ell}+v_{0,h}$$
  with
  $$u_{0,\ell}:=\mathcal{F}^{-1}(\mathbf{1}_{\mathcal{B}(0, \rho)}\widehat{u_0}),\quad B_{0,\ell}:=\mathcal{F}^{-1}(\mathbf{1}_{\mathcal{B}(0, \rho)}\widehat{B_0})\andf v_{0,\ell}:=\mathcal{F}^{-1}(\mathbf{1}_{\mathcal{B}(0, \rho)}\widehat{v_0}).$$
  Fix some $\eta\in(0,c)$ (with $c$ being the constant of \eqref{smallcon})
  and choose $\rho$  such that
  $$\|(u_{\ell,0}, B_{\ell,0}, v_{\ell,0})\|_{\dot{H}^\frac12}<\frac\eta2\cdotp$$
   By Theorem \ref{Th_1}, we know that there exists a unique global solution $(u_\ell, B_\ell)$
  to the Hall-MHD system supplemented  with data $(u_{\ell,0}, B_{\ell,0}),$ that  satisfies
\begin{equation}\label{lowsolu}
\|(u_\ell, B_\ell, v_\ell)\|_{L^\infty({\dot{H}^\frac12})}^2
+\frac12\|(u_\ell, B_\ell, v_\ell)\|_{L^2({\dot{H}^\frac32})}^2<\frac\eta2\with
v_\ell:=u_\ell-\nabla\times B_\ell\cdotp
\end{equation}
Let  $(u_h, B_h, v_h):=(u-u_\ell, B-B_\ell, v-v_\ell).$ We have $(u_h, B_h, v_h)\in\mathcal C(\mathbb R_+; \dot H^{\frac{1}{2}})\cap L^4(\mathbb R_+; \dot H^{1})$  since that result holds for both $(u, B, v)$ and $(u_\ell, B_\ell, v_\ell)$ (use interpolation). Furthermore,
 $(u_{0,h},\,B_{h,0}, J_{h,0})$ is in $L^2(\mathbb{R}^3)$ owing to the high-frequency cut-off
and we have  
\begin{equation}\label{3.d2}
\left\{
\begin{aligned}
 &\partial_t {u_h}-\Delta {u_h}:=\wt R_1,\\
 &\partial_t {B_h}-\Delta {B_h}:=\wt R_2,\\
 &\partial_t v_h-\Delta v_h:=\wt R_1+\wt R_3+\wt R_4+ \wt R_5,\\
 &(u_h, B_h, v_h)|_{t=0}{=}(u_{0,h},\,B_{h,0}, J_{h,0}),
\end{aligned}
\right.
\end{equation}
where
\begin{align*}
&\wt R_1:=\mathcal{P}(B\cdot\nabla B_h+B_h\cdot\nabla B_\ell-u\cdot\nabla u_h-u_h\cdot\nabla u_\ell),\\
&\wt R_2:=\nabla\times(v\times B_h+v_h\times B_\ell),\\
&\wt R_3:=-\nabla\times((\nabla\times v)\times B_h+(\nabla\times v_h)\times B_\ell),\\
&\wt R_4:=\nabla\times(v\times u_h+v_h\times u_\ell),\\
&\wt R_5:=2\nabla\times(v\cdot\nabla B_h+v_h\cdot\nabla B_\ell).
\end{align*}
Let us bound the terms $\wt R_1,$   $\wt R_2,$  $\wt R_4$ and  $\wt R_5$ as in  the proof of the uniqueness part of Theorem \ref{Th_1}, and  estimate 
 $\wt R_3$ as follows:
\begin{align*}
\|\wt R_3\|_{L^2_T(\dot{H}^{-1})}&\leq \|(\nabla\times v)\times B_h\|_{L^2_T(L^2)}+\|(\nabla\times v_h)\times B_\ell\|_{L^2_T(L^2)}\\
&\leq \|v\|_{L^4_T(\dot{H}^{1})} \|B_h\|_{L^4_T(L^\infty)}+\|v_h\|_{L^4_T(\dot{H}^{1})} \|B_\ell\|_{L^4_T(L^\infty)}.
\end{align*}
Note that our assumptions ensure that $B_\ell$ and $B_h$ are in $L^4(0, T; \dot{H}^{1}\cap \dot{H}^2),$ thus
 in $L^4(0, T; L^\infty)$ owing to the Gagliardo-Nirenberg inequality \eqref{eq:GN2} with $s=1$ and $s'=1.$
  Then one can conclude by a straightforward energy argument that
  $$\displaylines{\quad\|(u_h, B_h, v_h)(t)\|_{L^2}^2+2\int_0^t\|(u_h, B_h, v_h)\|_{\dot H^1}^2 \,d\tau
  \leq    \|(u_h, B_h, v_h)(0)\|_{L^2}^2
  \hfill\cr\hfill+C\int_0^t\wt V\|(u_h, B_h, v_h)\|_{L^2}^2\,d\tau\with
\wt V(t):=\|(u,u_\ell, B,B_\ell, v,v_\ell)(t)\|_{\dot{H}^{1}}^4\quad}$$
and  Gronwall's lemma  thus implies that 
$$\displaylines{\quad
\|(u_h, B_h, v_h)(t)\|_{L^2}^2+\int_0^t\|(u_h, B_h, v_h)(\tau)\|_{\dot H^1}^2\,d\tau\hfill\cr\hfill \leq\|(u_{h,0}, B_{h,0}, v_{h,0})\|_{L^2}^2\exp\Bigl(C{\int_0^t \wt V(\tau)\,d\tau}\Bigr)\cdotp\quad}
$$
Since $\wt V$ is globally integrable on $\R_+$ thanks to our assumptions and \eqref{lowsolu},
we see by interpolation that $(u_h,B_h,v_h)$ is in $L^4(\R_+;\dot H^{\frac12}).$ This in particular implies that
there exists some $t_0\geq0$ such that $\|(u_h, B_h, v_h)(t_{0})\|_{\dot H^\frac12}<\eta/2.$
 Hence $\|(u, B, v)(t_{0})\|_{\dot H^\frac12}<\eta$ and  Theorem \ref{Th_1} thus  ensures that 
$\|(u, B, v)(t)\|_{\dot H^\frac12}<\eta$ for all $t\geq t_0.$ This completes 
the proof of the corollary.  \end{proof}


\section{Weak-strong uniqueness}\label{s:weakstrong}

This section is devoted to the proof of  Theorem \ref{Th_3}. 
Let us underline that, in contrast with the other parts of the paper,  the proof works for any positive coefficients $\mu,$ $\nu$ and $\varepsilon.$
 Furthermore, it could be adapted to the 2$\frac12$D flows of the next section. 
 For expository purpose however, we   focus on the 3D case. 
\medbreak
Throughout, we shall repeatedly use the following result. 
\begin{lemma}\label{Le_P}
Let $a, b, c\in L^\infty(0,T; L^2(\R^3))\cap L^2(0,T; \dot{H}^1(\R^3))$ be three divergence free vector fields in $\R^3$. The following 
inequalities  hold:
\begin{itemize}
\item If, in addition,  $b$ belongs to $L^4(0, T; \dot{H}^1(\R^3)),$ then
\begin{equation}\label{Le_P1}
\Bigl|\int_0^T (a\cdot\nabla b\,|\, c)\,d\tau\Bigr|\lesssim \|a\|_{L^\infty_T(L^2)}^\frac12\|a\|_{L^2_T(\dot{H}^1)}^\frac12\|b\|_{L^4_T(\dot{H}^1)}\|c\|_{L^2_T(\dot{H}^1)}.
\end{equation}
\item If, in addition,   $c$ belongs to $L^4(0, T; \dot{H}^1(\R^3)),$ then
\begin{equation}\label{Le_P2}
\Bigl|\int_0^T (a\cdot\nabla b\,|\, c)\,d\tau\Bigr|\lesssim \|a\|_{L^\infty_T(L^2)}^\frac12\|a\|_{L^2_T(\dot{H}^1)}^\frac12\|b\|_{L^2_T(\dot{H}^1)}\|c\|_{L^4_T(\dot{H}^1)}.
\end{equation}
\item If, in addition,  $\nabla\times c$ belongs to $L^4(0, T; \dot{H}^1(\R^3)),$ then
\begin{equation}\label{Le_P3}
\Bigl|\int_0^T \bigl(\nabla\times ((\nabla\times a)\times b)\,|\, c\bigr)\,d\tau\Bigr|\lesssim \|a\|_{L^2_T(\dot{H}^1)}\|b\|_{L^\infty_T(L^2)}^\frac12\|b\|_{L^2_T(\dot{H}^1)}^\frac12\|\nabla\times c\|_{{L^4_T(\dot{H}^1)}}.
\end{equation}
\end{itemize}
\end{lemma}
\begin{proof}
To prove the first inequality, we use the fact that $a\cdot\nabla b=\div(b\otimes a)$  
and the duality inequality between $\dot H^1$ and $\dot H^{-1}$ so as to write 
$$\Bigl|\int_0^T (a\cdot\nabla b\,|\, c)\,d\tau\Bigr| \leq \int_0^T \|b\otimes a\|_{L^2} \|c\|_{\dot H^1}\,d\tau.$$
Hence, thanks to H\"older and Gagliardo-Nirenberg inequalities, and  Sobolev embedding \eqref{em},
\begin{align*}
\Bigl|\int_0^T(a\cdot\nabla b\,|\, c)\,d\tau\Bigr|&\leq \int_0^t \|a\|_{L^3}\|b\|_{L^6}\|\nabla c\|_{L^2}\,d\tau\\
&\lesssim \int_0^T \|a\|_{\dot{H}^\frac12}\|b\|_{\dot{H}^1}\|c\|_{\dot{H}^1}\,d\tau\\
&\lesssim  \|a\|_{L^4_T(\dot{H}^\frac12)}\|b\|_{L^4_T(\dot{H}^1)}\|c\|_{L^2_T(\dot{H}^1)}.
\end{align*}
Using the interpolation inequality \eqref{i1} yields \eqref{Le_P1}. 
\medbreak
Proving \eqref{Le_P2} is  similar. To get the last inequality, we take advantage of  \eqref{v1}, of  H\"older and Gagliardo-Nirenberg inequalities, and  Sobolev embedding:
\begin{align*}
\Bigl|\int_0^T \bigl(\nabla\times ((\nabla\times a)\times b)\,|\, c\bigr)\,d\tau\Bigr|&= \Bigl|\int_0^T \bigl((\nabla\times a)\times b\,|\, \nabla\times c\bigr)\,d\tau\Bigr|\\
&\leq  \int_0^T\|\nabla\times a\|_{L^2}\|b\|_{L^3}\|\nabla\times c\|_{L^6}\,d\tau\\
&\lesssim  \int_0^T \|a\|_{\dot{H}^1}\|b\|_{\dot{H}^\frac12}\|\nabla\times c\|_{\dot{H}^1}\,d\tau\\
&\lesssim \|a\|_{L^2_T(\dot{H}^1)}\|b\|_{L^4_T(\dot{H}^\frac12)}\|\nabla\times c\|_{{L^4_T(\dot{H}^1)}}.
\end{align*}
Using  the interpolation inequality \eqref{i1} completes the proof. 
\end{proof}
\medbreak
One can now start the proof of Theorem \ref{Th_3}. 
Let us recall our situation:   we are given two 
  Leray-Hopf solutions $(u, B)$ and $(\bar u,\bar B)$ in $L^\infty_T(L^2)\cap L^2_T(\dot{H}^1)$ 
  corresponding to the same initial  data  $(u_0, B_0)\in L^2$ with $\div u_0=\div B_0=0,$ 
  and assume in addition that   $u$ and  $J:=\nabla\times B$  (and thus also $B$) 
  are in $L^4_T(\dot{H}^1).$ We want to prove that the two solutions coincide on $[0,T],$
  that is to say   
  \begin{equation}\label{eq:zero}
  (\overline\du, \overline\dB)\equiv (0,0)\with (\overline\du, \overline\dB) :=(u-\bar u, B-\bar B).\end{equation}
    By definition of  what a  Leray-Hopf solution is,  both $(u,B)$ and $(\bar u,\bar B)$ 
    satisfy the energy inequality \eqref{energyineq} on $[0,T],$ which implies that  for all $t\geq T,$
\begin{multline}\label{equni}
\|(\overline{\du}, \overline{\dB})(t)\|_{L^2}^2+2\mu\int_0^t \|\nabla \overline{\du}\|_{L^2}^2\,d\tau+2\nu\int_0^t \|\overline{\dB}\|_{L^2}^2\,d\tau\\
\leq 2\|(u_0, B_0)\|_{L^2}^2-2(u(t)\,|\,\bar u(t))-2(B(t)\,|\,\bar B(t))\\-4\mu\int_0^t (\nabla u\,|\,\nabla\bar u)\,d\tau
-4\nu\int_0^t (\nabla B\,|\,\nabla\bar B)\,d\tau.\end{multline}
 Then, the  key to  proving  \eqref{eq:zero} is the following lemma, which is an adaptation to our setting 
 of a similar result for the Navier-Stokes equations in  \cite{Ga02}. 
 \begin{lemma}\label{L_4.2}
Under the assumptions of Theorem \ref{Th_3}, we have for all time $t\leq T,$
$$\displaylines{
(u(t)\,|\,\bar u(t))+(B(t)\,|\,\bar B(t))+2\mu\int_0^t(\nabla u\,|\,\nabla\bar u)\,d\tau+2\nu\int_0^t(\nabla B\,|\,\nabla\bar B)\,d\tau
\hfill\cr\hfill=\|u_0\|_{L^2}^2+\|B_0\|_{L^2}^2+\int_0^t\Bigl((\overline{\du}\cdot\nabla B\,|\,\overline{\dB})+(\overline{\du}\cdot\nabla u\,|\,\overline{\du})-(\overline{\dB}\cdot\nabla B\,|\,\overline{\du})\hfill\cr\hfill-(\overline{\dB}\cdot\nabla u\,|\,\overline{\dB})-\eps (\nabla\times((\nabla\times\overline{\dB})\times \overline{\dB})\,|\,B)\Bigr)d\tau.}$$
\end{lemma}
\begin{proof} The result is obvious  if $(u, B)$ and $(\bar u, \bar B)$ are smooth  and decay at infinity.
In our setting where the solutions are rough, it requires some justification. Therefore,
 we consider two sequences  $(u_n, B_n)_{n\in\N}$ and $(\bar u_n, \bar B_n)_{n\in\N}$ of smooth and  divergence free vector fields, such that
\begin{equation}\label{eq:ws1}\lim_{n\to\infty} (u_n, B_n,\nabla\times B_n)=(u, B,\nabla\times B)\quad{\rm{in}}\quad L^4_T(\dot{H}^1)\end{equation}
\begin{equation}\label{eq:ws3}\andf\lim_{n\to\infty} (u_n, B_n,\bar u_n, \bar B_n)=(u,B,\bar u, \bar B)\\\quad{\rm{in}}\quad L^2_T(\dot{H}^1)\cap L^\infty_T(L^2).\end{equation}
Since  our  assumptions  on  $(u,B)$ also 
ensure  that $(\partial_tu,\partial_tB)$ is in $L^2_T(\dot H^{-1}),$   one can  require 
in addition that 
\begin{equation}\label{eq:ws4}\lim_{n\to\infty} (\partial_t u_n,\partial_t B_n) =(\partial_tu,\partial_tB)
\quad{\rm{in}}\quad L^2_T(\dot H^{-1}).\end{equation}
Likewise, that $(\bar u,\bar B)$ is a Leray-Hopf solution guarantees that 
$(\partial_t\bar u,\partial_t\bar B)$ is in $L^{4/3}_T(\dot H^{-1}\times\dot H^{-2})$ 
(observe for exemple that $\bar u$ and $\bar B$ are in $L^{8/3}_T(L^4)$
and thus $u\otimes u$ and $B\otimes B$ are in $L^{4/3}_T(L^2)$ 
and similar properties for the other nonlinear terms of the Hall-MHD system). 
One shall thus require also that 
\begin{equation}\label{eq:ws2}  \lim_{n\to\infty} (\partial_t \bar u_n,\partial_t \bar B_n) =(\partial_t\bar u,\partial_t\bar B)
\quad{\rm{in}}\quad L^{4/3}_T(\dot H^{-1}\times \dot H^{-2}).\end{equation}
Now, taking  $(\bar u_n,\bar B_n)$ and $(u_n,B_n)$  as test functions in the 
weak formulation of \eqref{1.1}, \eqref{1.2}, \eqref{1.3} for $(u,B)$ and $(\bar u,\bar B),$ respectively, 
we get  for all $t\leq T,$
\begin{align}\label{eq:ws5a}
\int_0^t \Bigl((\partial_\tau  u\,|\, \bar u_n)+\mu (\nabla u\,|\,\nabla\bar u_n)+(u\cdot\nabla u\,|\, \bar u_n)-(B\cdot\nabla B\,|\, \bar u_n)\Bigl)\,d\tau&=0,\\\label{eq:ws5b}
\int_0^t \Bigl((\partial_\tau  \bar u\,|\,  u_n)+\mu (\nabla\bar u\,|\, \nabla u_n)+(\bar u\cdot\nabla\bar u\,|\, u_n)-(\bar B\cdot\nabla\bar B\,|\, u_n)\Bigl)\,d\tau&=0,\\
\label{eq:ws5c}
\int_0^t \Bigl((\partial_\tau  B\,|\,\bar B_n)+\nu (\nabla B\,|\, \nabla\bar B_n)+(u\cdot\nabla B\,|\, \bar B_n)-(B\cdot\nabla u\,|\, \bar B_n)\quad&\\+\eps(\nabla\times(J\times B)\,|\, \bar B_n)\Bigl)\,d\tau&=0,\\\label{eq:ws5d}
\int_0^t \Bigl((\partial_\tau  \bar B\,|\,  B_n)+\nu (\nabla\bar B\,|\,\nabla B_n)+(\bar u\cdot\nabla\bar B\,|\,B_n)-(\bar B\cdot\nabla\bar u\,|\,B_n)\quad&\\+\eps(\nabla\times(\bar J\times\bar B)\,|\, B_n)\Bigl)\,d\tau&=0.\end{align}
Since   $\nabla\bar u_n$ and $\nabla u_n$ converge to $\nabla\bar u$ and $\nabla u,$ 
in $L^2_T(L^2),$ we deduce that 
$$\lim_{n\to\infty}\biggl(\int_0^t(\nabla u\,|\nabla\bar u_n) \,d\tau+
\int_0^t(\nabla\bar u\,|\nabla u_n) \,d\tau\biggr)=2\int_0^t(\nabla u\,|\nabla\bar u) \,d\tau.$$
Thanks to \eqref{eq:ws1} and \eqref{eq:ws3}, and Lemma \ref{Le_P}, we have 
$$\lim_{n\to\infty}\int_0^t(u\cdot\nabla u\,|\, \bar u_n)\,d\tau=\int_0^t(u\cdot\nabla u\,|\, \bar u)\,d\tau,$$
$$\lim_{n\to\infty}\int_0^t(\bar u\cdot\nabla\bar u\,|\, u_n)\,d\tau=\int_0^t(\bar u\cdot\nabla\bar u\,|\, u)\,d\tau,$$
and one can pass to the limit similarly in all the quadratic terms that do not contain $J$ or $\bar J.$ 
Finally,  using the  following vector identity
\begin{equation}\label{eq:ws6}(a\times b)\cdot c=(c\times a)\cdot b=(b\times c)\cdot a,\end{equation}
Inequality \eqref{Le_P3} and \eqref{eq:ws1}, 
we get, since $\bar B_n$ is smooth, 
  $$\begin{aligned}
\int_0^t(\nabla\times( J\times B)\,|\,\bar B_n)(\tau)\,d\tau
&=\int_0^t(J\times B\,|\, \nabla\times\bar B_n)\,d\tau\\
&=\int_0^t(B\times (\nabla\times\bar B_n)\,|\, J)\,d\tau\\
&=-\int_0^t(\nabla\times((\nabla\times\bar B_n)\times B)\,|\, B)\,d\tau.
\end{aligned}$$
Hence, by \eqref{eq:ws3},  \eqref{Le_P3} and \eqref{eq:ws6},
$$\begin{aligned}
\lim_{n\to\infty}\int_0^t(\nabla\times( J\times B)\,|\, \bar B_n)\,d\tau&=-\int_0^t(\nabla\times((\nabla\times\bar B)\times B)\,|\,B)\,d\tau\\
&=\int_0^t(\nabla\times( J\times B)\,|\, \bar B)\,d\tau.
\end{aligned}$$
In order to prove that 
$$\lim_{n\to\infty}\int_0^t(\nabla\times(\bar J\times\bar B)\,|\, B_n)\,d\tau=\int_0^t(\nabla\times(\bar J\times\bar B)\,|\, B)\,d\tau,$$
one  may use directly  \eqref{Le_P3} and \eqref{eq:ws1}. 


In order to pass to the limit in the term of \eqref{eq:ws5a} with a time derivative, 
one may use \eqref{eq:ws3} and the fact that $\partial_tu$ is in $L^2_T(\dot H^{-1}).$
This gives
$$ \lim_{n\to\infty}\int_0^t(\partial_\tau u\,|\, \bar u_n)\,d\tau=\int_0^t (\partial_\tau u\,|\, \bar u)\,d\tau.$$
Next,  since   $\partial_t\bar u$ is in $L^{4/3}_T(\dot H^{-1}),$  
\eqref{eq:ws1} enables us to write that 
$$\lim_{n\to\infty}\int_0^t (\partial_\tau \bar u\,|\, u_n)\,d\tau=\int_0^t(\partial_\tau \bar u\,|\,u)\,d\tau.$$
In order to pass to the limit in the term of \eqref{eq:ws5c} with $\partial_tB,$ it
suffices to use the fact that $\partial_tB$ is in $L^2_T(\dot H^{-1})$
and \eqref{eq:ws3}. 
Passing  to the limit in the term of \eqref{eq:ws5c} with $\partial_t\bar B,$
relies on the property that $\partial_t\bar B$ is in $L^{\frac43}_T(\dot H^{-2})$
and on \eqref{eq:ws1}. 

Finally, passing to the limit in \eqref{eq:ws5a} and \eqref{eq:ws5b}, and adding up the 
resulting equalities yields for all $t\in[0,T],$ 
\begin{multline}\label{eq:ws7a}
\int_0^t \Bigl((\partial_\tau  u\,|\, \bar u)+ (\partial_\tau  \bar u\,|\,  u)\Bigr)d\tau
+2\mu \int_0^t (\nabla u\,|\,\nabla\bar u)\,d\tau\\
+\int_0^t\Bigl((u\cdot\nabla u\,|\, \bar u)-(B\cdot\nabla B\,|\, \bar u)+(\bar u\cdot\nabla\bar u\,|\, u)
-(\bar B\cdot\nabla\bar B\,|\, u)\Bigl)\,d\tau=0.\end{multline}
Applying the same procedure for  \eqref{eq:ws5a} and \eqref{eq:ws5b}, we get
\begin{multline}\label{eq:ws7b}
\int_0^t \Bigl((\partial_\tau  B\,|\, \bar B)+ (\partial_\tau  \bar B\,|\,  B)\Bigr)d\tau
+2\nu \int_0^t (\nabla B\,|\,\nabla\bar B)\,d\tau\\
+\int_0^t\Bigl((u\cdot\nabla B\,|\, \bar B)-(B\cdot\nabla u\,|\, \bar B)+(\bar u\cdot\nabla\bar B\,|\,B)-(\bar B\cdot\nabla\bar u\,|\,B)\Bigr)d\tau\\
+\eps\int_0^t\Bigl((\nabla\times(J\times B)\,|\, \bar B)
+(\nabla\times(\bar J\times\bar B)\,|\, B)\Bigr)d\tau=0.
\end{multline}
We claim that 
\begin{equation}\label{eq:ws8a}\int_0^t\bigl((\partial_\tau u\,|\, \bar u)+(u\,|\,\partial_\tau\bar u)\bigr)\,d\tau
=(u(t)\,|\,\bar u (t))- \|u_0\|_{L^2}^2.\end{equation}
Indeed, since both $u_n$ and $\bar u_n$ are smooth, we have 
$$\int_0^t\bigl((\partial_\tau u_n\,|\, \bar u_n)+(u_n\,|\,\partial_\tau\bar u_n)\bigr)\,d\tau
=(u_n(t)\,|\,\bar u_n(t))- (u_n(0)\,|\bar u_n(0)).$$
One can pass to the limit in the right-hand side thanks to \eqref{eq:ws3}. 
For the left-hand side, we write 
$$\int_0^t\bigl(\partial_\tau u_n\,|\, \bar u_n\bigr)\,d\tau 
- \int_0^t\bigl(\partial_\tau u\,|\, \bar u\bigr)\,d\tau=
\int_0^t\bigl(\partial_\tau u\,|\, (\bar u_n-\bar u)\bigr)\,d\tau
+\int_0^t(\partial_\tau(u_n-u)\,|\,\bar u_n)\,d\tau.$$
$$\int_0^t\bigl(u_n\,|\, \partial_\tau\bar u_n\bigr)\,d\tau 
- \int_0^t\bigl(u\,|\, \partial_\tau \bar u\bigr)\,d\tau=\int_0^t((u_n-u)\,|\,\partial_\tau\bar u)\,d\tau
+\int_0^t(u_n\,|\, \partial_\tau(\bar u_n-\bar u))\,d\tau.$$
We already proved that the first terms of the right-hand side converge to $0.$
For the second ones, this is due to \eqref{eq:ws4},\eqref{eq:ws2} and to the fact that 
$(\bar u_n)_{n\in\N}$ and $(u_n)_{n\in\N}$
are bounded in $L^2_T(\dot H^1)$ and $L^4_T(\dot H^1),$ respectively.
This proves \eqref{eq:ws8a}. 
\medbreak
In order to prove that 
\begin{equation}\label{eq:ws8b}
\int_0^t \Bigl((\partial_\tau B\,|\, \bar B)+(\partial_\tau \bar B\,|\, B)\Bigr)\,d\tau=(B(t)\,|\,\bar B(t))-\|B_0\|_{L^2}^2,\end{equation}
we start from the fact that 
$$\int_0^t\bigl((\partial_\tau B_n\,|\, \bar B_n)+(B_n\,|\,\partial_\tau\bar B_n)\bigr)\,d\tau
=(B_n(t)\,|\,\bar B_n(t))- (B_n(0)\,|\bar B_n(0)).$$
Passing to the limit in the right-hand side may be done thanks to 
 \eqref{eq:ws3}. 
For the left-hand side, we write 
$$\int_0^t\bigl(\partial_\tau B_n\,|\, \bar B_n\bigr)d\tau 
- \int_0^t\bigl(\partial_\tau B\,|\, \bar B\bigr)d\tau=
\int_0^t\bigl(\partial_\tau B\,|\, (\bar B_n-\bar B)\bigr)d\tau
+\int_0^t(\partial_\tau(B_n-B)\,|\,\bar B_n)\,d\tau,$$
$$\int_0^t\bigl(B_n\,|\, \partial_\tau\bar B_n\bigr)d\tau 
- \int_0^t\bigl(B\,|\, \partial_\tau \bar B\bigr)d\tau=\int_0^t((B_n-B)\,|\,\partial_\tau\bar B)d\tau
+\int_0^t(B_n\,|\, \partial_\tau(\bar B_n-\bar B))d\tau.$$
The convergence of the first terms of the right-hand side has already been shown before, 
and that of the second terms is due to the boundedness of $(B_n)_{n\in\N}$
and $(\bar B_n)_{n\in\N}$ in $L^4_T(\dot H^2)$ and $L^2_T(\dot H^1),$ respectively, 
and to \eqref{eq:ws4},\eqref{eq:ws2}. \medbreak
To conclude the proof of the lemma, one has to notice that
\begin{equation}\label{eq:ws9}
\int_0^t \bigl((u\cdot\nabla u\,|\, \bar u)+(\bar u\cdot\nabla\bar u\,|\, u)\bigr)\,d\tau=-\int_0^t (\overline{\du}
\cdot\nabla u\,|\,\overline{\du})\,d\tau,\end{equation}
\begin{equation}\label{eq:ws10}\int_0^t \bigl((u\cdot\nabla B\,|\, \bar B)+(\bar u\cdot\nabla\bar B\,|\,B)\bigr)\,d\tau=-\int_0^t (\overline{\du}\cdot\nabla\bar B\,|\, \overline{\dB})\,d\tau,\end{equation}
\begin{multline}\label{eq:ws11}
\int_0^t \Bigl((B\cdot\nabla B\,|\, \bar u)+(\bar B\cdot\nabla\bar u\,|\, B)+(\bar B\cdot\nabla\bar B\,|\, u)+(B\cdot\nabla u\,|\, \bar B)\Bigr)\,d\tau
\\=-\int_0^t \Bigl((\overline{\delta B}\cdot\nabla B\,|\, \overline{\du})+(\overline{\dB}\cdot\nabla u\,|\, \overline{\dB})\Bigr)\,d\tau,
\end{multline}
\begin{equation}\label{eq:ws12}
\int_0^t \!\Bigl((\nabla\!\times\!(J\!\times\! B)\,|\, \bar B)+(\nabla\!\times\!(\bar J\!\times\! \bar B)\,|\, B)\Bigr)d\tau
=\int_0^t \!(\nabla\times(\overline{\dJ}\times\overline{\dB})\,|\,B)\,d\tau.
\end{equation}
The above relations  are obvious in the case of smooth vector-fields
(just perform suitable integrations by parts and use the divergence-free property). 
Now, since the trilinear form corresponding to \eqref{Le_P1} is continuous 
on the spaces $$L^4_T(L^3)\times L^2_T(\dot H^1)\times L^4_T(L^6)
\andf L^4_T(L^3)\times L^4_T(\dot H^1)\times L^2_T(L^6),$$
we deduce that \eqref{eq:ws9}, \eqref{eq:ws10} and \eqref{eq:ws11} are still valid under our assumptions.

For justifying \eqref{eq:ws12} under our regularity framework, one just has to use the fact
that the trilinear form $(a,b,c)\mapsto \int_0^t (\nabla\times(a\times b)\,|\,c)\,d\tau$
is continuous on 
$$L^4_T(L^3)\times L^4_T(L^6)\times L^2_T(\dot  H^1)
\andf L^2_T(L^2)\times L^4_T(L^3)\times L^4_T(\dot  H^2)$$
and thus on  
$$
L^4_T(\dot H^{\frac12})\times L^4_T(\dot H^1)\times L^2_T(\dot  H^1)
\andf L^2_T(L^2)\times L^4_T(\dot H^{\frac12})\times L^4_T(\dot  H^2).$$
This completes the proof of Lemma \ref{L_4.2}.
\end{proof}
\medbreak
Let us finish  the proof of the  theorem.
Reformulating the right-hand side of  \eqref{equni} by means   Lemma \ref{L_4.2},  we get
\begin{multline}\label{eq:ws5}
\|(\overline{\du}, \overline{\dB})(t)\|_{L^2}^2+2\mu\int_0^t \|\nabla \overline{\du}\|_{L^2}^2\,d\tau+2\nu\int_0^t \|\nabla\overline{\dB}\|_{L^2}^2\,d\tau\\
\leq 2\int_0^t \Bigl((\overline{\dB}\cdot\nabla B\,|\,\overline{\du})+(\overline{\delta B}\cdot\nabla u\,|\,\overline{\dB})+\eps (\nabla\times((\nabla\times\overline{\dB})\times \overline{\dB})\,|\,B)\\
-(\overline{\du}\cdot\nabla B\,|\,\overline{\dB})-(\overline{\du}\cdot\nabla u\,|\,\overline{\du})\Bigr)\,d\tau.\end{multline}
Arguing as in  the proof of Lemma \ref{Le_P} and using  Young's inequality, we see  that
\begin{align*}
|(\overline{\dB}\cdot\nabla B\,|\,\overline{\delta u})|&\lesssim\|\overline{\dB}\|_{L^2}^\frac12\|\overline{\dB}\|_{\dot{H}^1}^\frac12\|B\|_{\dot{H}^1}\|\overline{\du}\|_{\dot{H}^1}\\
&\leq\frac{C}{\mu^2\nu}\|\overline{\dB}\|_{L^2}^2\|B\|_{\dot{H}^1}^4+\frac{\mu}{8}\|\overline{\delta u}\|_{\dot{H}^1}^2+\frac{\nu}{8}\|\overline{\dB}\|_{\dot{H}^1}^2,
\end{align*}
$$\begin{aligned}
|(\overline{\dB}\cdot\nabla u\,|\,\overline{\dB})|
&\leq\frac{C}{\nu^3}\|\overline{\dB}\|_{L^2}^2\|u\|_{\dot{H}^1}^4+\frac{\nu}{4}\|\overline{\dB}\|_{\dot{H}^1}^2,\\
|(\overline{\du}\cdot\nabla B\,|\,\overline{\dB})|
&\leq\frac{C}{\mu\nu^2}\|\overline{\du}\|_{L^2}^2\|B\|_{\dot{H}^1}^4+\frac{\mu}{8}\|\overline{\du}\|_{\dot{H}^1}^2+\frac{\nu}{8}\|\overline{\dB}\|_{\dot{H}^1}^2,\\
|(\overline{\du}\cdot\nabla u\,|\,\overline{\du})|
&\leq\frac{C}{\mu^3}\|\overline{\du}\|_{L^2}^2\|u\|_{\dot{H}^1}^4+\frac{\mu}{4}\|\overline{\du}\|_{\dot{H}^1}^2,
\end{aligned}$$
and
\begin{align*}
\eps|(\nabla\times((\nabla\times\overline{\dB})\times \overline{\dB})\,|\,B)|&\leq\eps \|\overline{\dB}\|_{L^2}^\frac12\|\overline{\delta B}\|_{\dot{H}^1}^\frac32\|\nabla\times B\|_{\dot{H}^1}\\
&\leq \frac{C\eps^4}{\nu^3}\|\overline{\dB}\|_{L^2}^2\|J\|_{\dot{H}^1}^4+\frac{\nu}{4}\|\overline{\dB}\|_{\dot{H}^1}^2.
\end{align*}
Thus, reverting to \eqref{eq:ws5}, we conclude that, for all $t\leq T,$
$$\displaylines{\quad
\|(\overline{\du}, \overline{\dB})(t)\|_{L^2}^2+\mu\int_0^t \|\nabla \overline{\du}\|_{L^2}^2\,d\tau+\nu\int_0^t \|\nabla\overline{\dB}\|_{L^2}^2\,d\tau\hfill\cr\hfill
\leq \frac{C\eps^4}{\min\{\mu, \nu\}^3}\int_0^t \|(\overline{\dB}, \overline{\du})\|_{L^2}^2\|(u, B, J)\|_{\dot{H}^1}^4\,d\tau,}$$
which, by Gronwall lemma, implies  that $(\overline{\du}, \overline{\dB})\equiv 0$ on $[0, T].$ 
\qed


\section{A global existence result for   \texorpdfstring{$2\frac12$D}{TEXT} flows}\label{s:21/2}

This section is devoted to the proof of Theorem \ref{Th_4}.
It essentially relies  on the following proposition 
and on an inequality for the vector-field $E$ defined in the introduction, 
that will proved at the end of the section. 
\begin{proposition}\label{p:sob2}
Let $(u, B)$ be a smooth solution of \eqref{1.1a}--\eqref{1.3a} with $\eps=\mu=\nu=1.$
Let $v:=u-j.$  
Then, we have 
\begin{equation}\label{6.1000}
\frac12\frac{d}{dt}\|(u, B)\|_{L^2}^2+\|(\wt\nabla u, \wt\nabla B)\|_{L^2}^2 = 0,
\end{equation}
 and  there exists a universal constant $C$ such that
\begin{equation}\label{6.11} 
\frac d{dt}\|v\|_{L^2}^2+\|\wt\nabla v\|_{L^2}^2\leq 
C\|(u, B, v)\|_{L^2}^2\|(\wt\nabla u, \wt\nabla B,\wt\nabla v)\|_{L^2}^2,
\end{equation}
\begin{equation}\label{6.15000}
\frac{d}{dt}\|B\|_{H^1}^2+\|\wt\nabla B\|_{H^1}^2
\leq C W\|B\|_{H^1}^2+C\|\wt\nabla B\|_{L^2}\|\wt\Delta B\|_{L^2}^2,
\end{equation}
with  $W(t):=\|u(t)\|_{L^2}^2\|\wt\nabla u(t)\|_{L^2}^2+ \|\wt\nabla u(t)\|_{L^2} \|\wt\nabla^2 u(t)\|_{L^2}.$
\end{proposition}
\begin{proof}
The first identity is just the energy balance. 
For proving \eqref{6.11}, we use the fact that the third equation of \eqref{1.100} rewrites for 2$\frac12$D flows a follows:
$$
\partial_tv-\wt\Delta v=\cP(\wt B\cdot\wt\nabla B-\wt u\cdot\wt\nabla u)
-\wt\nabla\times\bigl((\wt\nabla \times v)\times B\bigr)+\wt \nabla\times(v\times u)
+2\wt \nabla\times (\wt v\cdot\wt \nabla B).$$
Therefore, taking the $L^2$ scalar product with $v,$ integrating by parts in some terms, 
and using Cauchy-Schwarz inequality, one gets:
$$\frac12\frac d{dt}\|v\|_{L^2}^2+\|\wt\nabla v\|_{L^2}^2\leq 
\bigl(\|B\otimes B-u\otimes u\|_{L^2}+\|v\times u\|_{L^2}+2\| v\cdot\wt\nabla B\|_{L^2}\bigr)
\|\wt\nabla v\|_{L^2}.$$
Thanks to H\"older inequality, Sobolev embedding \eqref{em}, interpolation inequality \eqref{i1} and Young's inequality, we have
$$\begin{aligned}
\|B\otimes B\|_{L^2}\|\wt\nabla v\|_{L^2}&\leq \|B\|_{L^4}^2\|\wt\nabla v\|_{L^2}\\
&\leq C\|B\|_{L^2}\|\wt\nabla B\|_{L^2}\|\wt\nabla v\|_{L^2}\\
&\leq C\|B\|_{L^2}^2\|\wt\nabla B\|_{L^2}^2+\frac18\|\wt\nabla v\|_{L^2}^2,\end{aligned}$$
 a similar inequality for the term with $u\otimes u,$ 
$$\begin{aligned}
\|v\times u\|_{L^2}\|\wt\nabla v\|_{L^2}&\leq \|v\|_{L^4}\|u\|_{L^4}\|\wt\nabla v\|_{L^2}\\
&\leq C \|u\|_{L^2}\|v\|_{L^2}\|\wt\nabla u\|_{L^2}\|\wt\nabla v\|_{L^2} +\frac18 \|\wt\nabla v\|_{L^2}^2\\
&\leq C \|(u, v)\|_{L^2}^2\|(\wt\nabla u, \wt\nabla v)\|_{L^2}^2 +\frac18 \|\wt\nabla v\|_{L^2}^2,\end{aligned}$$
and, using  that $B=(-\wt\Delta)^{-1}\wt\nabla\times(u-v)$ and that $\wt\nabla^2(-\wt\Delta)^{-1}$ maps
$L^4$ to $L^4,$ 
$$\begin{aligned}
\|\wt v\cdot\wt\nabla B\|_{L^2}\|\wt\nabla v\|_{L^2}&\leq C\|v\|_{L^4}\|u-v\|_{L^4}\|\wt\nabla v\|_{L^2}\\
&\leq C\|v\|_{L^4}(\|u\|_{L^4}+\|v\|_{L^4})\|\wt\nabla v\|_{L^2}\\
&\leq C\|(u, v)\|_{L^2}^2\|(\wt\nabla u, \wt\nabla v)\|_{L^2}^2 +\frac14 \|\wt\nabla v\|_{L^2}^2.\end{aligned}$$
This yields \eqref{6.11}. 
\medbreak
For proving \eqref{6.15000},   use the following identity (valid if $\wt{\div}\wt y=\wt{\div}\wt z=0$):
\begin{equation}\label{I1}
\wt\nabla\times (y\times z)=\wt z\cdot\wt\nabla y-\wt y\cdot\wt\nabla z,
\end{equation}
to rewrite the equation for $B$ as follows: 
\begin{equation}\partial_t{\mathnormal B}-\wt\nabla\times(u\times B)+\eps\wt\nabla\times( j\times B)=\nu\wt\Delta B. \label{1.3b}
\end{equation}
Taking  the $L^2$ scalar product with $B$ yields
\begin{align*}
\frac12\frac{d}{dt}\|B\|_{L^2}^2+\|\wt\nabla B\|_{L^2}^2=(u\times B\,|\, j).
\end{align*}
To get an estimate for $\wt\nabla B,$  apply the following relations:  
$$\begin{aligned}
&\wt\nabla\times (y\times z)=(\wt\nabla\times y)\times z+(\wt\nabla\times z)\times y-2\wt y\cdot\wt\nabla z+\wt\nabla(y\cdot z),\\
&\wt\nabla\times(\wt\nabla\times y)+\wt\Delta y=0,\end{aligned}$$
so as to rewrite \eqref{1.3b} as
\begin{equation*}
\partial_tB+\wt u\cdot\wt\nabla B-\wt B\cdot\wt\nabla u-\wt\Delta B\times B-2\wt j\cdot\wt\nabla B+\wt\nabla(j\cdot B)=\wt\Delta B.
\end{equation*}
Taking the $L^2$ scalar product with $\wt\Delta B$  and using the fact that
$\wt\div\wt\Delta B=0,$  we get
\begin{equation*}
\frac12\frac{d}{dt}\|\wt\nabla B\|_{L^2}^2+\|\wt\Delta B\|_{L^2}^2=
(\wt B\cdot\wt\nabla u\,|\, \wt\Delta B)-
(\wt u\cdot\wt\nabla B\,|\,\wt\Delta B)-2(\wt j\cdot\wt\nabla B\,|\,\wt\Delta B).
\end{equation*}
Thanks to H\"older inequality,  \eqref{i1},  \eqref{em} and Young's inequality, we have
\begin{align*}
|(u\times B\,|\, j)|&\leq \|u\times B\|_{L^2}\|j\|_{L^2}\\
&\leq \|u\|_{L^4}\|B\|_{L^4}\|\wt\nabla B\|_{L^2}\\
&\leq C\|u\|_{L^2}^\frac12\|u\|_{\dot{H}^1}^\frac12\|B\|_{L^2}^\frac12\|\wt\nabla B\|_{L^2}^\frac32\\
&\leq C\|u\|_{L^2}^2\|\wt\nabla u\|_{L^2}^2\|B\|_{L^2}^2+\frac18\|\wt\nabla B\|_{L^2}^2\\
|(\wt u\cdot\wt\nabla B\,|\, \wt\Delta B)|
&\leq \|u\|_{L^4}\|\wt\nabla B\|_{L^4}\|\wt\Delta B\|_{L^2}\\
&\leq C\|u\|_{L^2}^\frac12\|u\|_{\dot{H}^1}^\frac12\|\wt\nabla B\|_{L^2}^\frac12\|\wt\nabla B\|_{\dot{H^1}}^\frac12\|\wt\Delta B\|_{L^2}\\
&\leq C\|u\|_{L^2}^2\|\wt\nabla u\|_{L^2}^2\|\wt\nabla B\|_{L^2}^2+\frac18\|\wt\Delta B\|_{L^2}^2,\\
|(\wt B\cdot\wt\nabla u\,|\, \wt\Delta B)|
&\leq \|B\|_{L^4}\|\wt\nabla u\|_{L^4}\|\wt\Delta B\|_{L^2}\\
&\leq C\|B\|_{L^2}^\frac12\|\wt\nabla B\|_{L^2}^\frac12\|\wt\nabla u\|_{L^2}^\frac12\|\wt\nabla^2 u\|_{L^2}^\frac12\|\wt\Delta B\|_{L^2}\\
&\leq C(\|B\|_{L^2}^2+\|\wt\nabla B\|_{L^2}^2)\|\wt\nabla u\|_{L^2}\|\wt\nabla^2 u\|_{L^2}
+\frac18\|\wt\Delta B\|_{L^2}^2,\\
|(\wt j\cdot\wt\nabla B\,|\,\wt\Delta B)|&\leq \|\wt j\cdot\wt\nabla B\|_{L^2}\|\wt\Delta B\|_{L^2}\\
&\leq \|j\|_{L^4}\|\wt\nabla B\|_{L^4}\|\wt\Delta B\|_{L^2}\\
&\leq C \|\wt\nabla B\|_{L^2} \|\wt\Delta B\|_{L^2}^2.
\end{align*}
Summing up the above estimates together yields \eqref{6.15000}.
\end{proof}
\medbreak
It is now easy to prove the first part of Theorem \ref{Th_4}:
adding up  \eqref{6.1000} and  \eqref{6.11}
yields  for some universal constant $C$ and all $t\geq0:$
$$\displaylines{\quad\|(u,B,v)(t)\|_{L^2}^2+\int_0^t\|(\wt\nabla u,\wt\nabla B,\wt\nabla v)\|_{L^2}^2\,d\tau
\hfill\cr\hfill\leq \|(u_0,B_0,v_0)\|_{L^2}^2+C\int_0^t\|(u,B,v)\|_{L^2}^2\|(\wt\nabla u,\wt\nabla B,\wt\nabla v)\|_{L^2}^2\,d\tau.\quad}$$
Lemma \ref{Le_7.000} (take $\alpha=2$, $W\equiv 0$) thus implies that 
if $ 2C\|(u_0,B_0,v_0)\|_{L^2}^2<1,$
then  we have for all time,  
\begin{equation}\label{eq:global}\|(u,B,v)(t)\|_{L^2}^2+\frac12\int_0^t\|(\wt\nabla u,\wt\nabla B,\wt\nabla v)\|_{L^2}^2\,d\tau\leq \|(u_0,B_0,v_0)\|_{L^2}^2.\end{equation}
{}From that stage, applying a regularization scheme similar to the one that
we used for handling the 3D case allows to conclude to the first 
part of Theorem \ref{Th_4} (uniqueness being also similar). 
\medbreak
In order to prove  the second part of the statement,   we observe that Inequality \eqref{6.15000} reads
$$\displaylines{
\frac d{dt}X^2+D^2\leq CXW +CXD^2\!\with\! 
X(t)=\|B(t)\|_{H^1},\!\quad\! D^2(t)=\int_0^t\|\wt\nabla B\|_{H^1}^2\,d\tau\hfill\cr\hfill
\andf W(t)=\|u(t)\|_{L^2}^2\|\wt\nabla u(t)\|_{L^2}^2+\|\wt\nabla u(t)\|_{L^2}\|\wt\nabla^2 u(t)\|_{L^2}.}$$
The first term of $W$ may be bounded thanks to  \eqref{energyineq}. 
To handle the second one, the idea is to get a bound for 
$\omega$ (that is the curl of $u$) through the identity \eqref{eq:E}. 
More precisely, taking  the scalar product
of \eqref{eq:E} with $E$ and integrating by parts,  we get (remember that $\mu=1$):
\begin{equation*}
\frac{1}{2}\frac{d}{dt}\|E\|_{L^2}^2+\|\wt\nabla E\|_{L^2}^2+(u\otimes \wt E\,|\, \wt\nabla E)=0.
\end{equation*}
Combining  H\"older and Gagliardo-Nirenberg inequalities  thus yields
\begin{align*}
\frac{1}{2}\frac{d}{dt}\|E\|_{L^2}^2+\|\wt\nabla E\|_{L^2}^2&\leq \|\wt E\|_{L^4}\|u\|_{L^4}\|\wt\nabla E\|_{L^2}\\
&\leq C\|E\|_{L^2}^\frac12\|u\|_{L^2}^\frac12\|\nabla u\|_{L^2}^\frac12\|\wt\nabla E\|_{L^2}^\frac32\\
&\leq  C\|E\|_{L^2}^2\|u\|_{L^2}^2\|\nabla u\|_{L^2}^2+\frac12\|\wt\nabla E\|_{L^2}^2.
\end{align*}
Taking advantage of  \eqref{energyineq} and using Gronwall lemma, we thus get
\begin{equation*}
\|E(t)\|_{L^2}^2
+\int_0^t\|\wt\nabla E\|_{L^2}^2\,d\tau\leq
\|E_0\|_{L^2}^2\exp\Bigl(C\|(u_0, B_0)\|_{L^2}^4\Bigr)\cdotp
\end{equation*} 
Since $\omega=E-B,$
using  again \eqref{energyineq}  eventually yields
\begin{equation}\label{energyomega}
\|\omega(t)\|_{L^2}^2+\int_0^t \|\wt\nabla \omega\|_{L^2}^2\,d\tau
\leq 2\|(\omega_0, B_0)\|_{L^2}^2\Bigl(1+\exp\Bigl(C\|(u_0, B_0)\|_{L^2}^4\Bigr)\Bigr)\cdotp\end{equation} 
Now, bounding $W$ according to  the energy balance \eqref{6.1000} and Inequality \eqref{energyomega}, 
and using Lemma \ref{Le_7.000} (take $\alpha=1$), one can conclude that, if 
$$C\|B_0\|_{H^1}\exp\Bigl(\|(u_0,B_0)\|_{L^2}^2 + (\|(u_0,B_0)\|_{L^2}^4+\|\omega_0+B_0\|_{L^2}^2\exp\bigl(C\|(u_0,B_0)\|_{L^2}^4\bigr)\Bigr)<1$$
then we have for all $t\geq0,$
$$\|B(t)\|_{H^1}^2+\int_0^t\|\nabla B\|_{H^1}^2\,d\tau \leq 1.$$
From that latter inequality,  \eqref{energyineq} and Inequality \eqref{energyomega}, 
one can work out a regularization procedure similar to that of Section \ref{s:trois}
and  complete the proof of the second part of Theorem \ref{Th_4}.\qed



\appendix
\section{}
For the reader's convenience, we here recall a few results  
that have been used repeatedly in the paper. Let us first recall
the definitions of Sobolev spaces and fractional derivation operators.
\begin{definition}
Let $s$ be in $\mathbb{R}.$ The \emph{homogeneous Sobolev space $\dot{H}^s(\mathbb{R}^d)$} (also denoted by $\dot{H}^s$) is the set of tempered distributions $u$ on $\mathbb{R}^d,$ with  Fourier transform in  $L^1_{loc}(\mathbb{R}^d),$  satisfying
$$\|u\|_{\dot{H}^s}:=\|\La^s u\|_{L^2}<\infty,$$
where   $\Lambda^s$  stands for  the fractional derivative  operator defined in terms of the Fourier transform by
\begin{equation*}
\mathcal F(\Lambda^s u)(\xi):=|\xi|^s\mathcal F u(\xi),\qquad \xi\in\R^d.
\end{equation*}
 The \emph{nonhomogeneous Sobolev space ${H}^s(\mathbb{R}^d)$} (also denoted by ${H}^s$) is the set of tempered distributions $u$ on $\mathbb{R}^d,$ with  Fourier transform in  $L^1_{loc}(\mathbb{R}^d),$  satisfying
$$\|u\|_{{H}^s}:=\|\langle D\rangle^s u\|_{L^2}<\infty\with 
\mathcal F(\langle D\rangle^s u)(\xi):=(1+|\xi|^2)^{s/2}\mathcal F u(\xi).$$
\end{definition}
We have the following proposition.
\begin{proposition}
Let $s_0\leq s\leq s_1.$ Then, $\dot{H}^{s_0}\cap\dot{H}^{s_1}$ is included in $\dot{H}^s,$ and we have
for all $\theta$ in $[0,1],$
\begin{equation}{\label{i1}}
\|u\|_{\dot{H}^s}\leq\|u\|_{\dot{H}^{s_0}}^{1-\theta}\|u\|_{\dot{H}^{s_1}}^{\theta}\quad{\rm{with}}\quad s=(1-\theta)s_0+\theta s_1.
\end{equation}
\end{proposition}
We also often used the following Sobolev embedding for  $0\leq s<d/2$:
\begin{equation}\label{em} \dot{H}^{s}(\mathbb{R}^d)\hookrightarrow L^{\frac{2d}{d-2s}}(\mathbb{R}^d)
\end{equation}
and the  Gagliardo-Nirenberg inequalities:
\begin{eqnarray}\label{eq:GN1} &&\|u\|_{L^p(\R^2)}\lesssim \|u\|_{L^2(\R^2)}^{\frac2p} \|\nabla u\|_{L^2(\R^2)}^{1-\frac2p},\quad
2\leq p<\infty,\\\label{eq:GN2}
&&\|u\|_{L^\infty(\R^3)}\lesssim \|u\|_{\dot H^s(\R^3)}^{1-\theta} \|u\|_{\dot H^{s'}(\R^3)}^\theta, \quad s<\frac32<s',\quad
\theta=\frac{\frac32-s}{s'-s}\cdotp\end{eqnarray}
Finally, we used the following inequalities (see e.g. \cite{Ke91}, Lemma 2.10):
\begin{lemma}\label{Le_28}
Let~$s>0$ and   $1< p, \, p_1,\, p_2,\, p_3,\, p_4 <\infty$ satisfying~$\frac{1}{p}=\frac{1}{p_1}+\frac{1}{p_2}=\frac{1}{p_3}+\frac{1}{p_4}\cdotp$
 There exists a constant $C>0$ such that
\begin{equation}
\|\Lambda^s (fg)-f\Lambda^s g\|_{L^p}\leq C(\|\nabla f \|_{L^{p_1}}\|\Lambda^{s-1}g\|_{L^{p_2}}+\|\Lambda^s f\|_{L^{p_3}}\|g\|_{L^{p_4}})
\end{equation}
{\rm{and}}
\begin{equation}
\|\Lambda^s (fg)\|_{L^p}\leq C(\|\Lambda^s f\|_{L^{p_1}}\|g\|_{L^{p_2}}+\|f\|_{L^{p_3}}\|\Lambda^s g\|_{L^{p_4}}).
\end{equation}
\end{lemma}
\medbreak
The following result has been used several times to establish global a priori estimates.

\begin{lemma}\label{Le_7.000}
Let $X,$ $D,$ $W$ be three   nonnegative  measurable functions on $[0,T]$ such that $X$ is also differentiable. Assume that there exist two  nonnegative real numbers~$C$ and~$\alpha$
such that   
\begin{equation}\label{7.000}
\frac{d}{dt}X^2+D^2\leq CWX^2+CX^\alpha D^2.
\end{equation}
If, in addition, 
\begin{equation}
2CX^\alpha(0)\exp\biggl(\frac {C\alpha}{2}\int_0^T W\,dt\biggr) < 1\label{7.200},
\end{equation}
then, for any $t\in[0,T]$, one has
\begin{equation}
X^2(t)+\frac12\int_0^t D^2\,d\tau\leq X^2(0)\exp\biggl(C\int_0^t W\,d\tau\biggr)\cdotp\label{7.300}
\end{equation}
\end{lemma}
\begin{proof}    Let $T^\star$ be the largest $t\leq T$ such that 
\begin{equation}\label{6.999} 2C\sup_{0\leq t'\leq t} X^\alpha(t')\leq1.\end{equation}
Then, \eqref{7.000} implies that for all $t\in[0,T^*],$ we have\begin{equation}
\frac{d}{dt}X^2+\frac12 D^2\leq  CWX^2.\label{7.400}
\end{equation}
By Gronwall  lemma, this  yields for all $t\in[0,T^\star],$ 
$$X^2(t)+\frac12\int_0^t D^2\,d\tau\leq X^2(0)\exp\biggl(C\int_0^tW\,d\tau\biggr)\cdotp$$
Hence,  it is clear that if \eqref{7.200} is satisfied, then \eqref{6.999} is satisfied with 
a strict inequality. A continuity argument thus ensures that we must have $T^\star=T$
and thus \eqref{7.300} on $[0,T].$ 
\end{proof}

\bigbreak\bigbreak

\noindent\textsc{Universit\'e Paris-Est Cr\'eteil,  LAMA UMR 8050, 61 avenue du G\'en\'eral de Gaulle,  94010 Cr\'eteil
and Sorbonne Universit\'e, LJLL UMR 7598, 4 Place Jussieu, 75005 Paris}\par\nopagebreak
E-mail address: raphael.danchin@u-pec.fr
\medbreak
\noindent\textsc{Universit\'e Paris-Est Cr\'eteil,  LAMA UMR 8050, 61 avenue du G\'en\'eral de Gaulle,  94010 Cr\'eteil}\par\nopagebreak
E-mail address: jin.tan@u-pec.fr
\end{document}